\documentclass[10pt]{article}
\usepackage{pdfsync}
\usepackage{epsfig}
\usepackage{amsmath}
\usepackage{amssymb}
\usepackage{graphicx}
\usepackage[latin1]{inputenc}
\usepackage{amsthm}
\usepackage[english]{babel}
{\catcode `\@=11 \global\let\AddToReset=\@addtoreset}
\AddToReset{equation}{section}

\newtheorem{lemma}{\bf Lemma}[section]

\newtheorem{property}{Property}[section]

\newtheorem{@definition}{\sc Definition}[section]

\newtheorem{@remark}{\sc Remark}[section]

\newtheorem{@example}{\sc Example}[section]

\newcommand{\beqn}{\begin{displaymath}}
\newcommand{\eeqn}{\end{displaymath}}
\newcommand{\beq}{\begin{equation}}  
\newcommand{\eeq}{\end{equation}}

\def\mathsf{\bf}
\def\N{\mathbb{N}}

\def\R{\mathbb{R}}
\def\Z{\mathbb{Z}}

\def\E{\mathrm E}

\def\text{\mbox}

\def\1{{\bf 1}}

\newcommand{\Cov}{\mbox{Cov}}

\newcommand{\Loi}{\mathcal{L}}

\def\limiteloiN{\renewcommand{\arraystretch}{0.5}
\begin{array}[t]{c}
\stackrel{{\Loi}}{\longrightarrow} \\
{\scriptstyle N\rightarrow\infty}
\end{array}\renewcommand{\arraystretch}{1}}

\def\limiteloiNm{\renewcommand{\arraystretch}{0.5}
\begin{array}[t]{c}
\stackrel{{\Loi}}{\longrightarrow} \\
{\scriptstyle [N/m]\wedge m \rightarrow\infty}
\end{array}\renewcommand{\arraystretch}{1}}

\def\limiteprobaN{\renewcommand{\arraystretch}{0.5}
\begin{array}[t]{c}
\stackrel{{\cal P}}{\longrightarrow} \\
{\scriptstyle N \rightarrow\infty}
\end{array}\renewcommand{\arraystretch}{1}}

\def\limiteN{\renewcommand{\arraystretch}{0.5}
\begin{array}[t]{c}
\stackrel{}{\longrightarrow} \\
{\scriptstyle N\rightarrow\infty}
\end{array}\renewcommand{\arraystretch}{1}}

\def\limitet{\renewcommand{\arraystretch}{0.5}
\begin{array}[t]{c}
\stackrel{}{\longrightarrow} \\
{\scriptstyle t\rightarrow\infty}
\end{array}\renewcommand{\arraystretch}{1}}

\def\limitet0{\renewcommand{\arraystretch}{0.5}
\begin{array}[t]{c}
\stackrel{}{\longrightarrow} \\
{\scriptstyle t\rightarrow 0}
\end{array}\renewcommand{\arraystretch}{1}}

\newtheorem{thm}{Theorem}
\newtheorem{rem}{Remark}

\newtheorem{prop}{Proposition}

\newtheorem{popy}{Property}

\def\Cov{\mathrm{Cov}}

\begin{document}

\title{\bf Semiparametric stationarity and fractional unit roots tests based on data-driven multidimensional increment ratio statistics}

\author{\centerline{Jean-Marc Bardet and B\'echir Dola} \\
\small {\tt bardet@univ-paris1.fr},~~ \small {\tt bechir.dola@univ-paris1.fr}\\
~\\
SAMM, Universit\'e Panth\'eon-Sorbonne (Paris I), 90 rue de Tolbiac, 75013 Paris, FRANCE}

\maketitle
\begin{abstract}
In this paper, we show that the central limit theorem (CLT) satisfied by the data-driven Multidimensional Increment Ratio (MIR) estimator of the memory parameter $d$ established in Bardet and Dola (2012) for  $d \in (-0.5,0.5)$ can be extended to a semiparametric class of Gaussian fractionally integrated  processes with memory parameter $d \in (-0.5,1.25)$. Since the asymptotic variance of this CLT can be estimated, by data-driven MIR tests for the two cases of stationarity and  non-stationarity, so two tests are constructed distinguishing the hypothesis $d<0.5$ and $d \geq 0.5$, as well as a fractional unit roots test distinguishing the case $d=1$ from the case $d<1$. Simulations done on numerous kinds of short-memory, long-memory and non-stationary processes, show both the high accuracy and robustness of this MIR estimator compared to those of usual semiparametric estimators. They also attest of the reasonable efficiency of MIR tests compared  to other usual stationarity tests or fractional unit roots tests.
\end{abstract}
\begin{quote}
\end{quote}

\vskip1cm

\section{Introduction}
The set $I(d)$ of fractionally integrated stochastic process $X=(X_k)_{k\in \Z}$ was defined and used in many articles (see for instance, Granger and Joyeux, 1980). Here we consider the following spectral version of this set for $-0.5 < d<1.5$:\\
~\\
{\bf Set $I(d)$:} {\em $X=(X_t)_{t\in \Z}$  is a stochastic process and there exists a continuous function $f^*:[-\pi,\pi] \to [0,\infty[$ satisfying:}
\begin{enumerate}
\item  {\em  if $-0.5< d<0.5$,  $X$ is a stationary process with a spectral density $f$ satisfying
$f(\lambda) = |\lambda|^{-2d}f^*(\lambda)$ for all $\lambda \in (-\pi,0)\cup (0,\pi)$, with $f^*(0)>0$}.
\item  {\em  if $0.5\leq d<1.5$, $U=(U_t)_{t\in \Z}=(X_{t}-X_{t-1})_{t\in \Z}$ is a stationary process with a spectral density $f$ satisfying $
f(\lambda) =  |\lambda|^{2-2d}f^*(\lambda)$ for all $\lambda \in (-\pi,0)\cup (0,\pi)$, with $f^*(0)>0$}.
\end{enumerate}
The case $d\in (0,0.5)$ is the case of long-memory processes, while $-0.5< d\leq 0$ corresponds to short-memory processes while $0.5 \leq d< 1.5$ corresponds to non-stationary processes having stationary increments. ARFIMA$(p,d,q)$ processes (which are linear processes), as well fractional Gaussian noises (with parameter $H=d+1/2\in (0,1)$) or fractional Brownian motions (with parameter $H=d-1/2\in (0,1)$) are famous examples of processes satisfying Assumption $I(d)$. The purpose of this paper is twofold: firstly, we establish the consistency of an adaptive data-driven  semiparametric estimator of $d$ for any $d\in (-0.5,1.25)$. Secondly, we use this estimator to build new stationarity and fractional unit roots semiparametric tests.  \\
~\\
Numerous articles have been devoted to the estimation of  $d$ in the case $d\in (-0.5,0.5)$ only. The books of Beran (1994) and Doukhan {\it et al.} (2003) provide large surveys of such parametric estimators (as maximum likelihood or Whittle estimators) or semiparametric estimators (as local Whittle, log-periodogram or wavelet based estimators). Here we will focus on the case of semiparametric estimators of processes satisfying Assumption $I(d)$. Even if first versions of local Whittle, log-periodogram and wavelet based estimators are considered in the case $d<0.5$ only (see for instance Robinson, 1995a and 1995b, Veitch {\it et al.}, 2003) , new extensions have been provided to estimate $d$ when $d\geq 0.5$ also (see for instance Hurvich and Ray, 1995, Velasco, 1999a,  Velasco and Robinson, 2000, Moulines and Soulier, 2003,  Shimotsu and Phillips, 2005, Giraitis {\it et al.}, 2003, 2006, Abadir {\it et al.}, 2007 or Moulines {\it et al.}, 2007). Moreover, adaptive data-driven versions of these estimators have been defined to avoid any trimming or bandwidth parameters, generally required by these methods (see for instance Giraitis {\it et al.}, 2000, Moulines and Soulier, 2003, Veitch {\it et al.}, 2003, or Bardet and Bibi, 2012). The first objective of this paper is to propose for the first time an adaptive data-driven estimator of $d$ satisfying a CLT, providing confidence intervals or tests, that is valid for $d<0.5$ but also for $d\geq 0.5$. This objective is achieved by using Multidimensional Increment Ratio (MIR) statistics. \\
The original version of the Increment Ratio (IR) statistic was defined in  Surgailis {\it et al.} (2008) from an observed trajectory $(X_1,\ldots,X_N)$ of a process $X$ satisfying $I(d)$ and for any $\ell \in \N^*$ as:
\begin{eqnarray}\label{defIR}
IR_N(\ell):=\frac{1}{N-3\ell} \,  \sum_{k=0}^{N-3\ell-1}\frac{\displaystyle \Big |\sum_{t=k+1}^{k+\ell}X_{t+\ell}-\sum_{t=k+1}^{k+\ell}X_{t}+\sum_{t=k+\ell+1}^{k+2\ell}X_{t+\ell}-\sum_{t=k+\ell+1}^{k+2\ell}X_{t}\Big |}{\displaystyle \Big |\sum_{t=k+1}^{k+\ell}X_{t+\ell}-\sum_{t=k+1}^{k+\ell} X_{t}\Big|+\Big|\sum_{t=k+\ell+1}^{k+2\ell} X_{t+\ell}-\sum_{t=k+\ell+1}^{k+2\ell}X_{t}\Big|}.
\end{eqnarray}
Under conditions on $X$, if $\ell \to \infty $ and $N /\ell \to \infty$, it is proved that the statistics $IR_N(\ell)$ converges to a deterministic monotone function $\Lambda_0(d)$ on $(-0.5,1.5)$ and a CLT is also established for $d \in (-0.5,0.5)\cup (0.5,1.25)$ when $\ell$ is large enough with respect to $N$. As a consequence of this CLT and using the Delta-method, the estimator $\widehat d_N(\ell)=\Lambda_0^{-1}(IR_N(\ell))$, where $d \mapsto \Lambda_0(d)$ is a smooth and increasing function defined in \eqref{DefinitionRhod}, is a consistent estimator of $d$ satisfying also a CLT (see more details below). However this new estimator was not totally satisfying. Firstly, it requires the knowledge of the second order behavior of the spectral density, which is clearly unknown in practice, to select $\ell$. Secondly, its numerical accuracy is reasonable but clearly lower than those of local Whittle or log-periodogram estimators. As a consequence, in Bardet and Dola (2012), we built a data-driven Multidimensional $IR$ (MIR) estimator $\widetilde d_N^{(MIR)}$ computed from $\big (\widehat d_N(\ell_1),\cdots,\widehat d_N(\ell_p)\big )$ (see its precise definition in \eqref{dtilde}) improving both these points but only for $-0.5 < d<0.5$. This is an adaptive data-driven semiparametric estimator of $d$ achieving the minimax convergence rate (up to a multiplicative logarithm factor) and requiring no regulation of any auxiliary parameter (as bandwidth or trimming parameters). Moreover, its numerical performances are comparable to the ones of local Whittle, log-periodogram or wavelet based estimators.  \\
Here we extend this previous work to the case $0.5\leq d <1.25$. Hence we obtain a CLT satisfied by $\widetilde d_N^{(MIR)}$ for all $d\in (-0.5,1.25)$ with an explicit asymptotic variance depending on $d$ only. This especially allows to obtain confidence intervals of $d$ using Slutsky Lemma. The case $d=0.5$ is now studied and this offers new perspectives: our data-driven estimator can be used for building a stationarity (or non-stationarity) test since $0.5$ is the ``border number'' between stationarity and non-stationarity. The case $d=1$ is also now studied and it provides another application of $\widetilde d_N^{(MIR)}$ to test fractional unit roots, that is to decide between $d=1$ and $d<1$.   ~\\
\\
There exist several famous stationarity (or non-stationarity) tests. We may cite parametric tests defined by Elliott {\it et al.} (1996) or Ng and Perron (1996, 2001). For non parametric stationarity tests we may cite the LMC test (see Leybourne and  McCabe, 2000) and the KPSS (Kwiatkowski, Phillips, Schmidt, Shin) test (see Kwiatkowski {\it et al.}, 1992), improved by the V/S test (see Giraitis {\it et al.}, 2003). For non-stationarity tests we may cite the Augmented Dickey-Fuller test (see Said and Dickey, 1984) and the Philipps and Perron test (PP test in the sequel, see Philipps and Perron, 1988).  All these tests are unit roots tests (except the V/S test which is also a short-memory test), which are, roughly speaking, tests based on the model $X_t=\rho\, X_{t-1} + \varepsilon_t$ with $|\rho|\leq 1$. A  right-tailed test $d\geq 0.5$ for a process satisfying Assumption $I(d)$ is therefore a refinement of a basic unit roots test since the case $\rho=1$ is a particular case of $I(1)$ and the case $|\rho|<1$ a particular case of $I(0)$. Thus, a stationarity (or non-stationarity
test) based on the estimator of $d$ provides a useful complementary test to usual unit roots tests. \\
This principle of stationarity test linked to $d$ has been already investigated in many articles. We can  cite Robinson (1994), Tanaka (1999), Ling and Li (2001), Ling (2003) or Nielsen (2004). It also be used to define fractional unit roots tests, like the Fractional Dickey-Fuller test defined by Dolado {\it et al.} (2002) or the cointegration rank test defined by Breitung {\it et al.} (2002). However, all these papers provide parametric tests, with a specified model (for instance ARFIMA or ARFIMA-GARCH processes). Extensions proposed by Lobato an Velasco (2007) and Dolado {\it et al.} (2008) allow to extend these tests to I$(d)$ processes with ARMA component but requiring the knowledge of the order of this component. Several papers have been recently devoted to the construction of semiparametric tests, see for instance Giraitis {\it et al.} (2006), Abadir {\it et al.} (2007) or Surgailis {\it et al.} (2008).  But these semiparametric tests require the knowledge of the second-order expansion of the spectral density at the zero frequency for adjusting a trimming or a bandwidth parameter; an a priori choice of this parameter always implies a bias of the estimator and therefore of the test when this asymptotic expansion is not smooth enough. \\
The MIR estimator $\widetilde d_N^{(MIR)}$ does not present this drawback. It converges to $d$ following a CLT with minimax convergence rate without any {\it a priori} choice of a parameter. This result is established for time series belonging to the Gaussian semiparametric class $IG(d,\beta)$ defined below (see the beginning of Section \ref{MIR}) which is a restriction of the general set $I(d)$. As a consequence, we construct a stationarity test $\widetilde S_N$ which accepts the stationarity assumption when $\widetilde d_N^{(MIR)}\leq 0.5+s$ with $s$ a threshold only depending on the type I error test, $\widetilde d_N^{(MIR)}$ and $N$. A non-stationarity test $\widetilde T_N$ accepting the non-stationarity assumption when $\widetilde d_N^{(MIR)}\geq 0.5-s$ is also proposed.
By the same principle, $\widetilde d_N^{(MIR)}$ also provides a fractional unit roots test $\widetilde F_N$ for deciding between $d=1$ and $d<1$, {\it i.e.} whether $\widetilde F_N \geq 1-s'$ or not, where $s'$ is a threshold depending on the type I error test. {\. 
~\\
In Section \ref{simu}, numerous simulations are realized on several models of time series (short and long-memory processes).
First, the new MIR estimator  $\widetilde d_N^{(MIR)}$ is compared to the most efficient and famous semiparametric estimators for several values of $d \in (-0.5,1.25)$. The performances of $\widetilde d_N^{(MIR)}$  are convincing: this estimator is accurate and robust for all the considered processes and is globally as efficient as local Whittle, log-periodogram or wavelet based estimators.
Secondly, the new stationarity $\widetilde S_N$ and non-stationarity $\widetilde T_N$ tests are compared to the most famous unit roots tests (KPSS, V/S, ADF and PP tests) for numerous I$(d)$ processes. And the results are quite surprising: even on AR$(1)$ or ARIMA$(1,1,0)$ processes, $\widetilde S_N$ and $\widetilde T_N$ tests provide convincing results which are comparable to those obtained with ADF and PP tests while those tests are especially built for these specific processes. For long-memory processes (such as ARFIMA processes), the results are clear: $\widetilde S_N$ and $\widetilde T_N$ tests are accurate tests of (non)stationarity while ADF and PP tests are only helpful when $d$ is close to $0$ or $1$. Concerning the new MIR fractional unit roots test $\widetilde F_N$, it provides satisfying results for all considered processes, while fractional unit roots tests such as the fractional Dickey-Fuller test developed by Dolado {\it et al.} (2002) or the efficient Wald test introduced by Lobato and Velasco (2007) are respectively only performing for ARFIMA$(0,d,0)$ processes or a class of long-memory processes containing ARFIMA$(p,d,0)$ processes but not ARFIMA$(p,d,q)$ processes with $q\geq 1$.    \\
~\\
The forthcoming Section~\ref{MIR} is devoted to the definition and asymptotic behavior of MIR estimators of $d$ and Section \ref{Adapt} studies an adaptive MIR estimator. The stationarity and non-stationarity tests are presented in Section \ref{test} while Section \ref{simu} deals with the results of simulations, Section \ref{conclu} provides conclusive remarks and Section \ref{proofs} contains all the proofs.

\section{The Multidimensional Increment Ratio statistic}\label{MIR}
Now we consider a semiparametric class $IG(d,\beta)$ which is a refinement of the general class $I(d)$. For $-0.5< d<1.5$ and $\beta>0$ define: \\
~\\
{\bf Assumption $IG(d,\beta)$:~} {\em $X=(X_t)_{t\in \Z}$  is a Gaussian process such that there exist $\epsilon >
0$, $c_0>0$, $c'_0>0$  and $c_1 \in \R$ satisfying:}
\begin{enumerate}
\item  {\em  if $d<0.5$,  $X$ is a stationary process with a spectral density $f$ satisfying for all $\lambda \in (-\pi,0)\cup (0,\pi)$}
\begin{eqnarray}\label{AssumptionS}
f(\lambda) =  c_{0}|\lambda|^{-2d}+c_{1}|\lambda|^{-2d+\beta} +O\big(|\lambda|^{-2d+\beta+\epsilon}\big) \quad \mbox{and}\quad |f'(\lambda)|\leq  c'_0 \, \lambda^{-2d-1}.
\end{eqnarray}
\item  {\em  if $0.5\leq d<1.5$, $U=(U_t)_{t\in \Z}=(X_{t}-X_{t-1})_{t\in \Z}$ is a stationary process with a spectral density $f$ satisfying for all $\lambda \in (-\pi,0)\cup (0,\pi)$}
\begin{eqnarray}\label{AssumptionAS}
f(\lambda) =  c_{0}|\lambda|^{2-2d}+c_{1}|\lambda|^{2-2d+\beta} +O\big(|\lambda|^{2-2d+\beta+\epsilon}\big) \quad \mbox{and}\quad |f'(\lambda)|\leq  c'_0 \, \lambda^{-2d+1}.
\end{eqnarray}
\end{enumerate}
Note that  Assumption $IG(d,\beta)$ is a particular (but still general) case of the set $I(d)$ defined above.
\begin{rem}\label{lin}
\begin{itemize}
\item The extension of the definition  from $d \in (-0.5,0.5)$ to $d\in [0.5,1.5)$ is classical since the conditions on the process is replaced by conditions on the process' increments.
\item The condition on the derivative $f'$ is not really usual. However, this is not a very restrictive condition since it is satisfied by all the classical long-range dependent processes.
\item In the literature, all the theoretical results concerning the IR statistic for time series have been obtained under Gaussian assumptions. In Surgailis {\it et al.} (2008) and Bardet and Dola (2012), simulations exhibited that the obtained limit theorems should be also valid for linear processes. However a theoretical proof of such result would require limit theorems for functionals of multidimensional linear processes difficult to be established, even if numerical experiments seem to show that this assumption could be replaced by the assumption that $X$ is a linear process having a fourth-moment order like it was done in Giraitis and Surgailis (1990).
\end{itemize}
\end{rem}
\noindent In this section, under Assumption $IG(d,\beta)$, we establish central limit theorems which extend to the case $d\in [0.5,1.25)$ those already obtained in Bardet and Dola (2012) for $d\in (-0.5,0.5)$. Let $X=(X_k)_{k\in \N}$ be a process satisfying Assumption $IG(d,\beta)$ and $(X_1,\cdots,X_N)$ be a path of $X$.
The statistic $IR_N$ (see its definition in \eqref{defIR})  was first defined in Surgailis {\it et al.} (2008) as a way to estimate the memory parameter. In Bardet and Surgailis (2011) a simple version of IR-statistic was also introduced to measure the roughness of continuous time processes, and its connection with level crossing index by geometrical arguments. The main interest of such a statistic is to be very robust to additional or multiplicative trends.  \\
~\\
As in Bardet and Dola (2012), let $m_{j}=j\, m,~j=1,\cdots,p$ with $p \in \N^*$ and $m\in \N^*$, and define the random vector $(IR_{N}(m_j))_{1\leq j\leq p}$.
In the sequel we naturally extend the results obtained for $m\in \N^*$ to $m\in (0,\infty)$ by the convention: $(IR_{N}(j \, m))_{1\leq j\leq p}=(IR_{N}(j \, [m]))_{1\leq j\leq p}$ (which does not change the asymptotic results). \\
For $H\in (0,1)$, let $B_H=(B_H(t))_{t\in \R}$ be a standard fractional Brownian motion, {\it i.e.} a centered Gaussian process having stationary increments and such as $\Cov\big (B_H(t)\,,\, B_H(s)\big )=\frac 1 2\, \big (|t|^{2H} +|s|^{2H} -|t-s|^{2H}  \big )$. Now, using obvious modifications of Surgailis {\it et al.} (2008), for $d \in (-0.5,1.25)$ and $p\in \N^*$, define the stationary multidimensional centered Gaussian processes $\big (Z_d^{(1)}(\tau),\cdots, Z_d^{(p)}(\tau)\big)$ such as for $\tau \in \R$,
\begin{eqnarray}\label{DefZ2}
Z_{d}^{(j)}(\tau):=\left \{ \begin{array}{ll}
\displaystyle \frac{\sqrt{2d(2d+1)}} {\sqrt{|4^{d+0.5}-4|}} \,  \int_{0}^{1}\big ( B_{d-0.5}(\tau+s+j)-B_{d-0.5}(\tau+s)\big )ds&\mbox{if $d\in (0.5,1.25)$} \\
\displaystyle \frac{1} {\sqrt{|4^{d+0.5}-4|}} \,  \big ( B_{d+0.5}(\tau+2j)-2\, B_{d+0.5}(\tau+j)+B_{d+0.5}(\tau)\big )&\mbox{if $d\in (-0.5,0.5)$}
\end{array} \right ..
\end{eqnarray}
Using a continuous extension when $d\to 0.5$ of the covariance of $Z_{d}^{(j)}(\tau)$, we also define the stationary multidimensional centered Gaussian processes $\big (Z_{0.5}^{(1)}(\tau),\cdots, Z_{0.5}^{(p)}(\tau)\big)$ with covariance such as:
\begin{eqnarray*}
\Cov \big ( Z_{0.5}^{(i)}(0), Z_{0.5}^{(j)}(\tau)\big ):=\frac 1 {4 \, \log 2 } \, \big ( -h(\tau+i-j)+h(\tau+i)+h(\tau-j)-h(\tau)\big )\quad \mbox{for $\tau \in \R$},
\end{eqnarray*}
where $h(x)=\frac 1 2 \big (|x-1|^2\log|x-1|+|x+1|^2\log|x+1|-2|x|^2\log|x|\big )$ for $x\in \R$, using the convention $0 \times \log 0=0$. Now, we establish a multidimensional CLT satisfied by $(IR_{N}(j \, m))_{1\leq j\leq p}$ for all $d\in (-0.5,1.25)$:
\begin{prop}\label{MCLT}
Assume that Assumption $IG(d,\beta)$ holds with $-0.5\leq d<1.25$ and $\beta>0$. Then
\begin{eqnarray}\label{TLC1}
\sqrt{\frac{N}{m}}\Big (IR_N(j \, m)-\E \big [IR_N(j \, m)\big ]\Big )_{1\leq j \leq p}\limiteloiNm {\cal N}(0, \Gamma_p(d))
\end{eqnarray}
with $\Gamma_p(d)=(\sigma_{i,j}(d))_{1\leq i,j\leq p}$ where for $i,j\in \{1,\ldots,p\}$,
\begin{eqnarray}\label{ssigma}
\sigma_{i,j}(d):&=&\int_{-\infty}^{\infty}\Cov \Big (\frac{|Z^{(i)}_{d}(0)+Z^{(i)}_{d}(i)|}{|Z^{(i)}_{d}(0)|+|Z^{(i)}_{d}(i)|}, \frac{|Z^{(j)}_{d}(\tau)+Z^{(j)}_{d}(\tau+j)|}{|Z^{(j)}_{d}(\tau)|+|Z^{(j)}_{d}(\tau+j)|}\Big )d\tau.
\end{eqnarray}
\end{prop}
\noindent The proof of this proposition as well as all the other proofs can be found in Section \ref{proofs}. \\
~\\
In the sequel, we will assume  that $\Gamma_p(d)$ is a positive definite  matrix for all $d \in (-0.5,1.25)$. Extensive numerical experiments seem to give strong evidence of such a property. Now, the CLT \eqref{TLC1} can be used for estimating $d$. To begin with,
\begin{property}\label{devEIR}
Let $X$ satisfy Assumption $IG(d,\beta)$ with $0.5\leq d<1.5$ and $0<\beta\leq 2$. Then, there exists a non-vanishing constant $K(d,\beta)$ depending only on $d$ and $\beta$ such that for $m$ large enough,
\begin{eqnarray*}
\E \big [IR_N(m)\big ] =\left\{ \begin{array}{ll}  \Lambda_0(d)+K(d,\beta)\times m^{-\beta} \, \big (1 + o(1)\big )& \mbox {if}~\beta<1+2d \\
 \Lambda_0(d)+ K(0.5,\beta) \times m^{-2} \log m  \, \big (1 + o(1)\big )& \mbox {if}~\beta=2~\mbox{and}~d=0.5
 \end{array}\right .
\end{eqnarray*}
\begin{eqnarray}\label{DefinitionRhod} 
\mbox{with}\quad \Lambda_0(d)&:=&\Lambda(\rho(d))\quad \mbox{where} \quad \rho(d):=\left \{\begin{array}{ll}\displaystyle \frac{4^{d+1.5}-9^{d+0.5}-7}{2(4-4^{d+0.5})}& \mbox{for}\quad 0.5<d<1.5\\
\displaystyle \frac{9\log(3)}{8\log(2)}-2 & \mbox{for}\quad d=0.5 \end{array} \right. \\
\mbox{and} && \Lambda(r):=\frac{2}{\pi}\, \arctan\sqrt{\frac{1+r}{1-r}}+\frac{1}{\pi}\, \sqrt{\frac{1+r}{1-r}}\log(\frac{2}{1+r})\quad \mbox{for $|r|\leq 1$}.
\end{eqnarray}
\end{property}
\noindent Therefore by choosing $m$ and $N$ such as  $\big (\sqrt {N/m} \big )m^{-\beta}\log m \to 0$ when $m,N\to
\infty$, the term $\E \big [IR(jm)\big ]$ can be replaced by $\Lambda_0(d)$ in Proposition \ref{MCLT}. Then, using the Delta-method with the function $(x_i)_{1\leq i \leq p} \mapsto (\Lambda^{-1}_0(x_i))_{1\leq i \leq p}$ (the function $d \in (-0.5,1.5) \to \Lambda_0(d)$ is a ${\cal C}^\infty$ increasing function), we obtain:
\begin{thm}\label{cltnada}
Let $\widehat d_N(j \, m):=\Lambda_0^{-1}\big (IR_N(j \, m)\big )$ for $1\leq j \leq p$. Assume that Assumption $IG(d,\beta)$ holds with $0.5\leq d<1.25$ and $0<\beta\leq 2$. Then if $m \sim C\, N^\alpha$ with $C>0$ and $(1+2\beta)^{-1}<\alpha<1$,
\begin{eqnarray}\label{cltd}
\sqrt{\frac{N}{m}}\Big (\widehat d_N(j \, m)-d\Big )_{1\leq j \leq p}\limiteloiN {\cal N}\Big(0,(\Lambda'_0(d))^{-2}\, \Gamma_p(d)\Big ).
\end{eqnarray}
\end{thm}
\noindent This result is an extension to the case $0.5\leq d \leq 1.25$ from the case $-0.5<d<0.5$ already obtained in Bardet and Dola (2012). Note that the consistency of $\widehat d_N(j \, m)$ is ensured when $1.25\leq d <1.5$ but the previous CLT does not hold (the asymptotic variance of $\sqrt{\frac{N}{m}}\, \widehat d_N(j \, m)$ diverges to $\infty$ when $d >1.25$, see Surgailis {\it et al.}, 2008).\\
~\\
Now define
\begin{equation}\label{Sigma}
\widehat \Sigma_N(m):=(\Lambda'_0(\widehat d_N(m))^{-2}\, \Gamma_p(\widehat d_N(m)).
\end{equation}
The function $d\in (-0.5,1.5) \mapsto \sigma(d)/\Lambda'(d)$ is ${\cal C}^\infty$ and therefore, under assumptions of Theorem \ref{cltnada},
$$\widehat \Sigma_N(m) \limiteprobaN (\Lambda'_0(d))^{-2}\, \Gamma_p(d).$$
Thus, a pseudo-generalized least square estimation (PGLSE) of $d$ can be defined by
$$
\widetilde d_N(m):=\big (J_p^{\intercal} \big (\widehat \Sigma_N(m)\big )^{-1} J_p \big )^{-1}\, J_p^{\intercal} \,  \big ( \widehat \Sigma_N(m) \big )^{-1} \big (\widehat d_N(m_i) \big ) _{1\leq i \leq p}
$$
with $J_p:=(1)_{1\leq j \leq p}$ and denoting $J_p^{\intercal}$ its transpose.
From a Gauss-Markov Theorem type (see again Bardet and Dola, 2012), the asymptotic variance of $\widetilde d_N(m)$ is smaller than the one of any $\widehat d_N(jm)$, $j=1,\ldots,p$. Hence, we obtain under the assumptions of Theorem \ref{cltnada}:
\begin{eqnarray}\label{TLCdtilde}
\sqrt{\frac{N}{m}}\big (\widetilde d_N(m)-d\big ) \limiteloiN {\cal N}\Big(0 \,, \, \Lambda'_0(d)^{-2}\,\big (J_p^{\intercal} \, \Gamma^{-1}_p(d)J_p\big )^{-1}\Big ).
\end{eqnarray}

\section{The adaptive data-driven version of the estimator}\label{Adapt}
Theorem \ref{cltnada} and CLT \eqref{TLCdtilde} require the knowledge of $\beta$ to be applied. But in practice $\beta$ is unknown. The procedure defined in Bardet and Bibi (2012) or Bardet and Dola (2012) can be used for obtaining a data-driven selection of an optimal sequence $(\widetilde m_N)$ derived from an estimation of $\beta$. Since the case $d \in (-0.5,0.5)$ was studied in Bardet and Dola (2012) we consider here $d \in [0.5,1.25)$ and for $\alpha \in (0,1)$, define
\begin{equation}\label{QNdef}
Q_N(\alpha,d):=\big (\widehat d_N(j\, N^\alpha)-\widetilde d_N(N^{\alpha}) \big )^{\intercal}_{1\leq j \leq p} \big (\widehat \Sigma_N(N^\alpha)\big )^{-1}\big (\widehat d_N(j\, N^\alpha)-\widetilde d_N(N^{\alpha}) \big )_{1\leq j \leq p},
\end{equation}
which corresponds to the sum of the pseudo-generalized squared distance between the points $(\widehat d_N(j\, N^\alpha))_j$ and PGLSE of $d$.
Note that by the previous convention, $\widehat d_N(j\, N^\alpha)=\widehat d_N(j\, [N^\alpha])$ and $\widetilde d_N(N^\alpha)=\widetilde d_N([N^\alpha])$. Then $\widehat Q_N(\alpha)$ can be minimized on a discretization of $(0,1)$ and define:
\begin{eqnarray*}
\widehat \alpha_N :=\mbox{Argmin}_{\alpha \in {\cal A}_N}
\widehat Q_N(\alpha)\quad \mbox{with}\quad {\cal A}_N=\Big \{\frac {2}{\log
N}\,,\,\frac { 3}{\log N}\,, \ldots,\frac {\log [N/p]}{\log N}
\Big \}.
\end{eqnarray*}
\begin{rem}\label{defAn}
The choice of the set of discretization ${\cal A}_N$ is implied by our proof of convergence of $\widehat \alpha_N$. If the interval $(0,1)$ is stepped in $N^c$ points, with $c>0$, the used proof cannot
attest this convergence. However $\log N$ may be replaced in the previous expression of ${\cal A}_N$ by any negligible function
of $N$ compared to functions $N^c$ with $c>0$ (for instance, $(\log N)^a$ or $a\log N$ with $a>0$ ).
\end{rem}
\noindent From the central limit theorem ({\ref{cltd}}) one deduces the following limit theorem:
\begin{prop}\label{hatalpha}
Assume that Assumption $IG(d,\beta)$ holds with $0.5\leq d<1.25$ and $0<\beta\leq 2$. Then,
$$
\widehat \alpha_N
\limiteprobaN \alpha^*=\frac 1 {(1+2\beta)}.
$$
\end{prop}
\noindent
Finally define
$$
\widetilde
m_N:=N^{\widetilde \alpha_N}\quad \mbox{with}\quad \widetilde \alpha_N:=\widehat \alpha_N+ \frac {6\, \widehat \alpha_N} {(p-2)(1-\widehat
{\alpha}_N )} \cdot \frac {\log \log N}{\log N}.
$$
and the estimator
\begin{equation}\label{dtilde}
\widetilde d_N^{(MIR)} :=\widetilde
d_N(\widetilde m_N)=\widetilde d_N(N^{\widetilde \alpha_N}).
\end{equation}
(the definition and use of $\widetilde \alpha_N$ instead of $\widehat \alpha_N$ are explained just before Theorem 2 in Bardet and Dola, 2012).  The following theorem provides the asymptotic behavior of the estimator $\widetilde d_N^{(MIR)}$:
\begin{thm}\label{tildeD}
Under assumptions of Proposition \ref{hatalpha},
\begin{eqnarray}\label{CLTD2}
&&\sqrt{\frac{N}{N^{\widetilde \alpha_N }}} \big(\widetilde d_N^{(MIR)}  - d \big)
\limiteloiN    {\cal N}\Big (0\, ; \,\Lambda'_0(d)^{-2}\,\big (J_p^{\intercal} \, \Gamma^{-1}_p(d)J_p\big )^{-1}\Big ).
\end{eqnarray}
Moreover, $\displaystyle ~~\forall
\rho>\frac {2(1+3\beta)}{(p-2)\beta},~\mbox{}~~ \frac {N^{\frac
{\beta}{1+2\beta}} }{(\log N)^\rho} \cdot \big|\widetilde d_N^{(MIR)} - d \big|
\limiteprobaN  0.$
\end{thm}
\noindent The convergence rate of $\widetilde d_N^{(MIR)}$ is the same (up to a multiplicative logarithm factor) than the one of minimax estimator of $d$ in this semiparametric framework (see Giraitis {\it et al.}, 1997). As it was already established in Surgailis {\it et al.} (2008), the use of IR statistics confers a robustness of $\widetilde d_N^{(MIR)}$ to smooth additive or multiplicative trends (see also the results of simulations thereafter). The additional advantage of $\widetilde d_N^{(MIR)}$ with respect to other adaptive estimators of $d$ (see Moulines and Soulier, 2003, for an overview over frequency domain estimators of $d$) is the central limit theorem (\ref{CLTD2}) satisfied by $\widetilde d_N^{(MIR)}$. This central limit theorem provides asymptotic confidence intervals on $d$ which are unobtainable for instance with FEXP or local periodogram adaptive estimator (see respectively Iouditsky {et al.}, 2001, and Giraitis {\it et al.}, 2000 or Henry, 2007).
Moreover $\widetilde d_N^{(MIR)}$ can be used for $d\in (-0.5,1.25)$, {\it i.e.} as well for stationary and non-stationary processes, without modifications in its definition. Both these advantages allow to define stationarity and fractional unit roots tests based on $\widetilde d_N^{(MIR)}$.

\section{Stationarity, non-stationarity and fractional unit roots tests}\label{test}
Assume that $(X_1,\ldots,X_N)$ is an observed trajectory of a process $X=(X_k)_{k\in \Z}$. We define here new stationarity, non-stationarity and fractional unit roots tests for $X$ based on $\widetilde d_N^{(MIR)}$.
\subsection{A stationarity test}
There exist many stationarity and non-stationarity tests. The most famous stationarity tests are certainly the following unit roots tests:
\begin{itemize}
\item The KPSS (Kwiatkowski, Phillips, Schmidt, Shin) test  (see Kwiatkowsli {\it et al.}, 1992);
\item The V/S test (see its presentation in Giraitis {\it et al.}, 2001) which was first defined for testing the presence of long-memory versus short-memory. As it was already notified in Giraitis {\it et al.} (2003-2006), the V/S test is also more powerful than the KPSS test for testing the stationarity.
\item A test based on unidimensional IR statistic and developed in Surgailis {\it et al.} (2008).
\end{itemize}
More precisely, we consider here the following statistical hypothesis test:
\begin{itemize}
\item \underline{Hypothesis $H_0$} (stationarity):  $(X_t)_{t\in \Z}$ is a process satisfying Assumption $IG(d,\beta)$ with $d\in (-0.5,0.5)$ and $0<\beta\leq 2$.
\item \underline{Hypothesis $H_1$} (non-stationarity):  $(X_t)_{t\in \Z}$ is a process satisfying Assumption $IG(d,\beta)$ with $d\in [0.5,1.25) $ and $0<\beta\leq 2$.
\end{itemize}
We use a test based on $\widetilde d_N^{(MIR)}$ for deciding between both these hypothesis. Hence from the previous CLT \eqref{CLTD2} and with a significance level $\alpha$, define
\begin{equation}\label{SNtild}
\widetilde S_N:= \1_{ \widetilde{d}_{N}^{(MIR)}>0.5+\sigma_p(0.5)\,   q_{1-\alpha} \,   N^{(\widetilde{\alpha}_N-1)/2}},
\end{equation}
where $\sigma_p(0.5)=\Big ( \Lambda'_0(0.5)^{-2}\,\big (J_p^{\intercal} \, \Gamma^{-1}_p(0.5)J_p\big )^{-1}\Big )^{1/2}$(see \eqref{CLTD2}) and $q_{1-\alpha}$ is the $(1-\alpha)$ quantile of a standard Gaussian random variable ${\cal N}(0,1)$. \\
~\\
Then we define the following rules of decision:
\begin{center} "$H_0$ (stationarity) is accepted when $\widetilde S_N=0$ and rejected when $\widetilde S_N=1$."
\end{center}
\begin{rem}
In fact, the previous stationarity test $\widetilde S_N$ defined in \eqref{SNtild} can also be seen as a semiparametric test $d<d_0$ versus $d \geq d_0$ with $d_0=0.5$. It is obviously possible to extend it to any value $d_0 \in (-0.5,1.25)$ by defining
$
\widetilde S^{(d_0)}_N:= \1_{ \widetilde{d}_{N}^{(MIR)}>d_0+\sigma_p(d_0)\,   q_{1-\alpha} \,   N^{(\widetilde{\alpha}_N-1)/2}}.
$
The particular case $d_0=1$ will be considered thereafter as a fractional unit roots test.
\end{rem}
From previous results, it is clear that:
\begin{popy} \label{StatIG}
Under Hypothesis $H_0$, the asymptotic type I error of the test $\widetilde S_N$ is $\alpha$ and under Hypothesis $H_1$, the test power tends to $1$.
\end{popy}
Moreover, this test can be used as a unit roots (UR) test. Indeed, define the following typical problem of UR test. Let $X_t=at+b+\varepsilon_t$, with $(a,b)\in \R^2$, and $\varepsilon_t$ an ARIMA$(p,d,q)$ with $d=0$ or $d=1$. Then, a (simplified)  problem of a UR test is to decide between:
\begin{itemize}
\item $H^{UR}_0$: $d=0$ and $(\varepsilon_t)$ is a stationary ARMA$(p',q')$ process.
\item $H^{UR}_1$: $d=1$ and $(\varepsilon_t-\varepsilon_{t-1})_t$ is a stationary ARMA$(p',q')$ process.
\end{itemize}
Then,
\begin{popy}\label{StatUR}
Under Hypothesis $H^{UR}_0$, the type I error of this unit roots test problem using $\widetilde S_N$ decreases to $0$ when $N \to \infty$ and under Hypothesis $H^{UR}_1$, the test power tends to $1$.
\end{popy}

\subsection{A non-stationarity test}
Unit roots tests are also often used as non-stationarity test. Hence, between the most famous non-stationarity tests and in a nonparametric framework, consider
\begin{itemize}
\item The Augmented Dickey-Fuller (ADF) test (see Said and Dickey, 1984);
\item The Philipps and Perron (PP) test (see for instance Phillips and Perron 1988).
\end{itemize}
Using the statistic $\widetilde{d}_{N}^{(MIR)}$ we propose a new non-stationarity test $\widetilde T_N$ for deciding between:
\begin{itemize}
\item \underline{Hypothesis $H'_0$} (non-stationarity):  $(X_t)_{t\in \Z}$ is a process satisfying Assumption $IG(d,\beta)$ with $d\in [0.5,1.25)$  and $\beta \in (0,2]$.
\item \underline{Hypothesis $H'_1$} (stationarity):  $(X_t)_{t\in \Z}$ is a process satisfying Assumption $IG(d,\beta)$ with $-0.5 < d<1/2$ and $\beta \in (0,2]$.
\end{itemize}
Then, the decision rule of the test under the significance level $\alpha$ is the following:
\begin{center}
"Hypothesis $H'_0$ is accepted when $\widetilde T_N=1$ and rejected when $\widetilde T_N=0$"
\end{center}
where
\begin{equation}\label{TNtild}
\widetilde T_N:= \1_{ \widetilde{d}_{N}^{(MIR)}<0.5-\sigma_p(0.5)\,   q_{1-\alpha} \,   N^{(\widetilde{\alpha}_N-1)/2}}.
\end{equation}
Then,
\begin{popy} \label{NStatIG}
Under Hypothesis $H_0'$, the asymptotic type I error of the test $\widetilde T_N$ is $\alpha$ and under Hypothesis $H_1'$ the test power tends to $1$.
\end{popy}
As previously, this test can also  be used as a unit roots test where $X_t=at+b+\varepsilon_t$, with $(a,b)\in \R^2$, and $\varepsilon_t$ an ARIMA$(p,d,q)$ with $d=0$ or $d=1$. We consider here a ``second'' simplified problem of unit roots test which is to decide between:
\begin{itemize}
\item $H^{UR'}_0$: $d=1$ and $(\varepsilon_t-\varepsilon_{t-1})_t$ is a stationary ARMA$(p',q')$ process.
\item $H^{UR'}_1$: $d=0$ and $(\varepsilon_t)_t$ is a stationary ARMA$(p',q')$ process.
\end{itemize}
Then,
\begin{popy}\label{StatUR2}
Under Hypothesis $H^{UR'}_0$, the type I error of the unit roots test problem using $\widetilde T_N$ decreases to $0$ when $N \to \infty$ and under Hypothesis $H^{UR'}_1$ the test power tends to $1$.
\end{popy}

\subsection{A fractional unit roots test}
Fractional unit roots tests have also been defined for specifying the eventual long-memory property of the process in a unit roots test. In our Gaussian framework, they consist on testing
\begin{itemize}
\item \underline{Hypothesis $H^{FUR}_0$}:  $(X_t)_{t\in \Z}$ is a "random walk"-type process such as:
\begin{equation}\label{H0FUR}
X_t=X_{t-1}+u_t
\end{equation}
with $(u_t)_t$ a process satisfying Assumption $IG(0,\beta)$ with   $0<\beta\leq 2$. Therefore $(X_t)$ is a process satisfying Assumption $IG(1,\beta)$.
\item \underline{Hypothesis $H^{FUR}_1$} :  $(X_t)_{t\in \Z}$ is a process satisfying the following relation:
\begin{equation}\label{H1FUR}
X_t=X_{t-1}+\phi \, \Delta^{d_1} X_{t-1}+u_t
\end{equation}
where $(u_t)_t$ is a process satisfying Assumption $IG(0,\beta)$ with   $0<\beta\leq 2$, $\phi<0$, and $\Delta ^{d_1}$ is the fractional integration operator of order $0<d_1<1$, {\it i.e.} $ \Delta ^{d_1}X_{t-1}=\sum_{i=0}^{t-1} \pi_i(d_1) X_{t-1-i}$ and $\pi_i(d_1)=\Gamma(i-d_1)\big (\Gamma(i+1)\Gamma(-d_1) \big) ^{-1}$.
\end{itemize}
After computations, it follows that if $X$ satisfies \eqref{H1FUR}, then $X$ satisfies Assumption $IG(d_1,\beta)$.
There exist several fractional unit roots tests (see for example, Robinson, 1994, Tanaka, 1999, Dolado {\it et al.}, 2002, or more recently, Kew and Harris, 2009).
It is clear that the estimator $\widetilde d_N^{(MIR)}$ can be used in such a framework for testing fractional unit roots by comparing $\widetilde d_N^{(MIR)}$ to $1$. Hence, the decision rule of the test under the significance level $\alpha$ is the following:
\begin{center}
"Hypothesis $H^{FUR}_0$ is accepted when $\widetilde F_N=1$ and rejected when $\widetilde F_N=0$"
\end{center}
where
\begin{equation}\label{FNtild}
\widetilde F_N:= \1_{ \widetilde{d}_{N}^{(MIR)}>1-\sigma_p(1)\,   q_{1-\alpha} \,   N^{(\widetilde{\alpha}_N-1)/2}}.
\end{equation}
Then as previously
\begin{popy} \label{FURIG}
Under Hypothesis $H^{FUR}_0$, the asymptotic type I error of the test $\widetilde F_N$ is $\alpha$ and under Hypothesis $H_1^{FUR}$ the test power tends to $1$.
\end{popy}

\section{Results of simulations}\label{simu}
\subsection{Numerical procedure for computing the estimator and tests}
First of all, softwares used in this Section are available on {\tt
http://samm.univ-paris1.fr/-Jean-Marc-Bardet} with a free access on (in Matlab language).\\
~\\
The concrete procedure for applying the MIR-test of stationarity is the following:
\begin{enumerate}
 \item using additional simulations (performed on ARMA, ARFIMA, FGN processes and not presented here in order to avoid overloading the paper), we have observed that the value of the parameter $p$ is not really important with respect to the accuracy of the test (there are less than $10\%$ of fluctuations on the value of $\widetilde d^{(MIR)}_N$ when $p$ varies). However, for optimizing our procedure (in the sense of minimizing from simulation the mean square error of the $d$ estimation) we chose $p$ as a stepwise function of $N$:
$$
p=5\times \1_{\{N<120\}}+ 10\times \1_{\{120\leq N<800\}}+ 15\times \1_{\{800\leq N<10000\}}+20\times \1_{\{N\geq 10000\}}.
$$
\item as the values of $\sigma_p(0.5)$ and $\sigma_p(1)$ are essential for computing the thresholds of the tests, we have estimated them and obtained:
\begin{itemize}
\item $\sigma_{5}(0.5)\simeq 0.9082, ~\sigma_{10}(0.5) \simeq 0.8289, ~ \sigma_{15}(0.5) \simeq 0.8016~\mbox{and}~  \sigma_{20}(0.5)\simeq 0.7861.
$
\item $\sigma_{5}(1)\simeq 0.8381, ~\sigma_{10}(1) \simeq 0.8102, ~ \sigma_{15}(1) \simeq 0.8082~\mbox{and}~  \sigma_{20}(1)\simeq 0.7929.
$
\end{itemize}
\item then after computing $\widetilde m_N$ presented in Section \ref{Adapt}, the adaptive estimator $\widetilde d^{(MIR)}_N$ defined in \eqref{dtilde}, the test statistics  $\widetilde{S}_{N}$ defined in \eqref{SNtild}, $\widetilde{T}_{N}$ defined in \eqref{TNtild} and $\widetilde{F}_{N}$ defined in \eqref{FNtild} are computed.
\end{enumerate}
\subsection{Monte-Carlo experiments on several time series}
In the sequel the results are obtained from $1000$ generated independent trajectories of
each process defined below. The concrete
procedures of generation of these processes are obtained from the
circulant matrix method, as detailed in Doukhan {\it et al.} (2003).
The simulations are realized for different values of $d$ and
$N$ and processes which satisfy Assumption $IG(d,\beta)$:
\begin{enumerate}
\item the usual ARIMA$(p',d,q')$ processes with respectively $d=0$ or $d=1$ and an innovation process which is a Gaussian white noise.
Such processes satisfy
{Assumption} $IG(0,2)$ or $IG(1,2)$ (respectively);
\item the ARFIMA$(p',d,q')$ processes with parameter $d$ such that
$d \in (-0.5,1.25)$ and an innovation process which is a Gaussian white noise.
Such ARFIMA$(p',d,q')$  processes satisfy Assumption $IG(d,2)$ (note that ARIMA processes are particular cases of ARFIMA processes).

\item the Gaussian stationary processes $X^{(d,c_1,d)}$ with the spectral density
\begin{eqnarray}
f_3(\lambda)=\frac 1 {|\lambda|^{2d}}(1+c_1\, |\lambda|^{\beta})~~~\mbox{for
$\lambda \in [-\pi,0)\cup (0,\pi]$},
\end{eqnarray}
with $d \in (-0.5,1.5)$, $c_1>0$ and $\beta\in (0,\infty)$. Therefore the spectral density $f_{3}$ implies that
Assumption $IG(d,\beta)$ holds. In the sequel we will first use $c_1=1$ and $\beta=0.1$, implying that the second order term of the spectral density is "less negligible" than in case of ARFIMA processes, and $c_1=0$, implying that the second order term of the spectral density is "more negligible" than in case of ARFIMA processes.
\item the Gaussian stationary processes $X^{(d,\log)}$, such as its spectral density is
\begin{eqnarray}
f_4(\lambda)=\frac 1 {|\lambda|^{2d}}(1+ |\log (\lambda) |\, |\lambda|)~~~\mbox{for
$\lambda \in [-\pi,0)\cup (0,\pi]$},
\end{eqnarray}
with $d \in (-0.5,1.5)$. Therefore the spectral density $f_{4}$ implies that
Assumption $I(d)$ holds, but not  $IG(d,\beta)$ {\it stricto sensu}.
\item the Gaussian non-stationary process $X^{(trend)}$ which can be written as $   X^{(trend)}_t=a_n(t)+\sigma_n(t) \times ARFIMA(0,d,0)$, where the additive  and multiplicative trends are respectively $a_n(t)=\sin(2\pi t/n) $ and $\sigma_n(t)=\sqrt {2t/n}$ (for us we chose a non-polynomial  but smooth additive trend).
\end{enumerate}

\subsubsection{Comparison of $\widetilde d_N^{(MIR)}$  with other semiparametric estimators of $d$}
Here we first compare the performance of the data-driven MIR estimator $\widetilde d_N^{(MIR)}$ with other famous semiparametric estimators of $d$:
\begin{itemize}
\item $\widehat d_N^{(IR)}$ is the original version of the IR based estimator defined in Surgailis {\it et al.} (2008). As it was recommended in that article, we chose $m=10$.
\item $\widetilde{d}_N^{(MS)}$ is the global log-periodogram estimator introduced
by Moulines and Soulier (2003), also called FEXP estimator,
with bias-variance balance parameter $\kappa=2$. Such an estimator was shown to be consistent for $d\in (-0.5,1)$. This semiparametric estimator is an adaptive data-driven estimator of $d$.
\item $\widehat d_N^{(ADG)}$ is the extended local Whittle estimator defined by
Abadir, Distaso and Giraitis (2007) which is consistent for $d>-3/2$. It is a generalization of the local Whittle estimator introduced by Robinson (1995b), consistent for $d<0.75$, following a first extension proposed by Phillips (1999)  and Shimotsu and Phillips (2005). This estimator avoids the tapering used for instance in Velasco (1999b) or Hurvich and Chen (2000). The trimming parameter is chosen as $m=N^{0.65}$ (this is not an adaptive data-driven estimator) following the numerical recommendations of Abadir {\it et al.} (2007).
\item $\widetilde{d}_N^{(WAV)}$ is an adaptive data-driven wavelet based estimator
introduced in Bardet and Bibi (2012) using a Lemarie-Meyer type wavelet (another similar choice could be the adaptive wavelet estimator introduced
in Veitch {\it et al.}, 2003, using a Daubechie's wavelet, but its robustness property are slightly less interesting). The asymptotic normality of such estimator is established for $d>-0.5$ (when the number of vanishing moments of the wavelet function is large enough).
\end{itemize}
Note that  only $\widehat{d}_N^{(IR)}$ and  $\widehat d_N^{(ADG)}$ are not data-driven adaptive among the $5$ estimators.
Table \ref{Table1} provides the results of simulations for ARIMA$(1,d,0)$ ($N=500$, $N=5000$ and $N=50000$). For ARFIMA$(0,d,0)$, ARFIMA$(1,d,1)$, $X^{(d,1,1)}$, $X^{(d,0,1)}$, $X^{(d,\log)}$ and $X^{(trend)}$ processes and several values of $d$, the results of simulations are presented  for $N=500$ (Table \ref{Table2}), $N=5000$ (Table \ref{Table3}) and $N=50000$ (Table \ref{Table4}).  \\
~\\
\begin{table}[t]
{\footnotesize
\begin{center}
\begin{tabular}{|c|c|c|c|c|c||c|c|c|c|c|}
\hline\hline
  $N=500$  &$d=0$ & $d=0$ &  $d=0$ & $d=0$ & $d=0$  &$d=1$ & $d=1$  & $d=1$ & $d=1$  & $d=1$ \\
ARIMA$(1,d,0)$ &$\phi$=-0.1 &$\phi$=-0.3 &$\phi$=-0.5 &$\phi$=-0.7 &$\phi$=-0.9 &$\phi$=-0.1 &$\phi$=-0.3  &$\phi$=-0.5  &$\phi$=-0.7 &$\phi$=-0.9 ~ \\
\hline \hline
$\sqrt{MSE}$ $\widetilde{d}_N^{(MIR)}$ & 0.1022& 0.1174  &  0.1617   & 0.2507   & 0.6114   & 0.1088   & 0.1133   & 0.1313& 0.1954 & 0.3625
\\
\hline
$\sqrt{MSE}$ $\widehat{d}_N^{(IR)}$  & 0.2169   & 0.2341   & 0.2456  &0.3061 & 0.6209 &0.1590    & 0.1589  & 0.1535 & 0.1587 & 0.2859
\\
\hline
$\sqrt{MSE}$ $\widetilde{d}_N^{(MS)}$  &0.1424 & 0.1407 &  0.1463  & 0.1539   & 0.4141   & 0.1721   & 0.1699   & 0.1655 & 0.1849 & 0.3298
\\
\hline
$\sqrt{MSE}$ $\widehat{d}_N^{(ADG)}$&  0.0785 & 0.0847 &  0.1244   & 0.2649   & 0.6755    & 0.0787   & 0.0805   & 0.1205 & 0.2633 & 0.4985
\\
\hline
$\sqrt{MSE}$ $\widetilde{d}_N^{(WAV)}$ & 0.0676 & 0.1397 & 0.2408    & 0.4106   & 0.7531   & 0.0717   & 0.0980   & 0.1341 & 0.1785 & 0.3883
\\
\hline
\hline
    $N=5000$  &$d=0$ & $d=0$& $d=0$ & $d=0$ & $d=0$  &$d=1$ & $d=1$  & $d=1$ & $d=1$  & $d=1$\\
ARIMA$(1,d,0)$ &$\phi$=-0.1 &$\phi$=-0.3 &$\phi$=-0.5 &$\phi$=-0.7 &$\phi$=-0.9 &$\phi$=-0.1 &$\phi$=-0.3  &$\phi$=-0.5  &$\phi$=-0.7 &$\phi$=-0.9  \\
\hline \hline
$\sqrt{MSE}$ $\widetilde{d}_N^{(MIR)}$  & 0.0340 & 0.0553  & 0.0759   & 0.1024   & 0.2893   & 0.0344   & 0.0479   & 0.0614 & 0.0872 & 0.2679
\\
\hline
$\sqrt{MSE}$ $\widehat{d}_N^{(IR)}$   & 0.0678  & 0.0804 & 0.1150 &  0.2290   & 0.6041    & 0.0479   & 0.0477  & 0.0524 & 0.0900 & 0.2721
 \\
\hline
$\sqrt{MSE}$ $\widetilde{d}_N^{(MS)}$ & 0.0412 & 0.0440   & 0.0421   & 0.0441   & 0.2223   & 0.0422   & 0.0454  & 0.0470 & 0.0462 & 0.1533
\\
\hline
$\sqrt{MSE}$ $\widehat{d}_N^{(ADG)}$ & 0.0321 & 0.0337   & 0.0372   & 0.0816   & 0.3751   & 0.0318   & 0.0335   & 0.0398 & 0.0817 & 0.3724
\\
\hline
$\sqrt{MSE}$ $\widetilde{d}_N^{(WAV)}$& 0.0376 & 0.0625    & 0.0716   & 0.0970   & 0.2144   & 0.0344   & 0.0496   & 0.0552 & 0.0700 & 0.1245
\\
\hline
\hline
    $N=50000$ &$d=0$ & $d=0$ & $d=0$ & $d=0$ & $d=0$  &$d=1$ & $d=1$  & $d=1$ & $d=1$  & $d=1$ \\
ARIMA$(1,d,0)$ &$\phi$=-0.1 &$\phi$=-0.3 &$\phi$=-0.5 &$\phi$=-0.7 &$\phi$=-0.9 &$\phi$=-0.1 &$\phi$=-0.3  &$\phi$=-0.5  &$\phi$=-0.7 &$\phi$=-0.9  \\
\hline \hline
$\sqrt{MSE}$ $\widetilde{d}_N^{(MIR)}$  & 0.0152 & 0.0246  & 0.0303   & 0.0435   & 0.0756   & 0.0109   & 0.0179   & 0.0274 & 0.0348 & 0.0560
\\
\hline
$\sqrt{MSE}$ $\widehat{d}_N^{(IR)}$   & 0.0238  & 0.0469 & 0.0945 &  0.2124   & 0.5985    & 0.0170   & 0.0182  & 0.0303 & 0.0827 & 0.2745
 \\
\hline
$\sqrt{MSE}$ $\widetilde{d}_N^{(MS)}$ & 0.0140 & 0.0140   & 0.0137   & 0.0143   & 0.1489   & 0.0130   & 0.0149  & 0.0142 & 0.0158 & 0.1005
\\
\hline
$\sqrt{MSE}$ $\widehat{d}_N^{(ADG)}$ & 0.0152 & 0.0130  & 0.0145   & 0.0220   & 0.1418   & 0.0165   & 0.0159   & 0.0148 & 0.0231 & 0.1396
\\
\hline
$\sqrt{MSE}$ $\widetilde{d}_N^{(WAV)}$& 0.0190 & 0.0171    & 0.0366   & 0.0353   & 0.0568   & 0.0227   & 0.0224  & 0.0291 & 0.0458 & 0.0517
\\
\hline
\end{tabular}
~\\
\vspace{0.7cm}
\begin{tabular}{|c|c|c|c|c|c||c|c|c|c|c|}
\hline\hline
  $N=500$  &$d=0$ & $d=0$ &  $d=0$ & $d=0$ & $d=0$  &$d=1$ & $d=1$  & $d=1$ & $d=1$  & $d=1$ \\
ARIMA$(1,d,0)$ &$\phi$=0.1~ &$\phi$=0.3~ &$\phi$=0.5 ~&$\phi$=0.7~ &$\phi$=0.9 ~&$\phi$=0.1~ &$\phi$=0.3 ~ &$\phi$=0.5~  &$\phi$=0.7 ~&$\phi$=0.9 ~ \\
\hline \hline
$\sqrt{MSE}$ $\widetilde{d}_N^{(MIR)}$ & 0.0995 & 0.1020 &  0.1115  & 0.1280  & 0.1165   & 0.1065   & 0.1102  & 0.1131& 0.1161 & 0.1155
\\
\hline
$\sqrt{MSE}$ $\widehat{d}_N^{(IR)}$  & 0.2093   & 0.2017   & 0.2069  &0.2096 & 0.1878 &0.1632   & 0.1649  & 0.1587 & 0.1611 & 0.1658
\\
\hline
$\sqrt{MSE}$ $\widetilde{d}_N^{(MS)}$  &0.1478 & 0.1382 &  0.1430  & 0.1401  & 0.1521   & 0.1649   & 0.1658   & 0.1658 & 0.1827 & 0.2006
\\
\hline
$\sqrt{MSE}$ $\widehat{d}_N^{(ADG)}$&  0.0809 & 0.0776 &  0.0808   & 0.0820   & 0.0765    & 0.0807   & 0.0809  & 0.0843 & 0.0809 & 0.0825
\\
\hline
$\sqrt{MSE}$ $\widetilde{d}_N^{(WAV)}$ & 0.0994 & 0.1214 & 0.1278  & 0.1257   & 0.1247   & 0.0875  & 0.1054  & 0.1058 & 0.1042 & 0.1002
\\
\hline
\hline
    $N=5000$  &$d=0$ & $d=0$& $d=0$ & $d=0$ & $d=0$  &$d=1$ & $d=1$  & $d=1$ & $d=1$  & $d=1$\\
ARIMA$(1,d,0)$ &$\phi$=0.1 &$\phi$=0.3 &$\phi$=0.5 &$\phi$=0.7 &$\phi$=0.9 &$\phi$=0.1 &$\phi$=0.3  &$\phi$=0.5  &$\phi$=0.7 &$\phi$=0.9  \\
\hline \hline
$\sqrt{MSE}$ $\widetilde{d}_N^{(MIR)}$  & 0.0354 & 0.0499  & 0.0726   & 0.0950   & 0.0897   & 0.0430   & 0.0437   & 0.0442 & 0.0492 & 0.0546
\\
\hline
$\sqrt{MSE}$ $\widehat{d}_N^{(IR)}$   & 0.0669  & 0.0743 & 0.0918 &  0.1271   & 0.1031   & 0.0503   & 0.0490 & 0.0486 & 0.0490 & 0.0522
 \\
\hline
$\sqrt{MSE}$ $\widetilde{d}_N^{(MS)}$ & 0.0435 & 0.0450   & 0.0422  & 0.0423   & 0.0518   & 0.0451   & 0.0450  & 0.0443 & 0.0451 & 0.0566
\\
\hline
$\sqrt{MSE}$ $\widehat{d}_N^{(ADG)}$ & 0.0360 & 0.0350  & 0.0324   & 0.0328   & 0.0337   & 0.0335   & 0.0341   & 0.0334 & 0.0338 & 0.0334
\\
\hline
$\sqrt{MSE}$ $\widetilde{d}_N^{(WAV)}$& 0.0405 & 0.0486    & 0.0469   & 0.0472   & 0.0480   & 0.0369   & 0.0511  & 0.0448 & 0.0484 & 0.0451
\\
\hline
\hline
    $N=5000$ &$d=0$ & $d=0$ & $d=0$ & $d=0$ & $d=0$  &$d=1$ & $d=1$  & $d=1$ & $d=1$  & $d=1$ \\
ARIMA$(1,d,0)$ &$\phi$=0.1 &$\phi$=0.3 &$\phi$=0.5 &$\phi$=0.7 &$\phi$=0.9 &$\phi$=0.1 &$\phi$=0.3  &$\phi$=0.5  &$\phi$=0.7 &$\phi$=0.9  \\
\hline \hline
$\sqrt{MSE}$ $\widetilde{d}_N^{(MIR)}$  & 0.0161 & 0.0280  & 0.0435   & 0.0672   & 0.0727  & 0.0143   & 0.0152  & 0.0163& 0.0162& 0.0191
\\
\hline
$\sqrt{MSE}$ $\widehat{d}_N^{(IR)}$   & 0.0217 & 0.0407 & 0.0662 &  0.1130   & 0.0880    & 0.0148   & 0.0160  & 0.0189 & 0.0176 & 0.0193
 \\
\hline
$\sqrt{MSE}$ $\widetilde{d}_N^{(MS)}$ & 0.0146 & 0.0150   & 0.0153   & 0.0162   & 0.0261   & 0.0132   & 0.0136  & 0.0137 & 0.0148 & 0.0209
\\
\hline
$\sqrt{MSE}$ $\widehat{d}_N^{(ADG)}$ & 0.0158 & 0.0151  & 0.0144   & 0.0147   & 0.0136   & 0.0133   & 0.0142   & 0.0132 & 0.0140 & 0.0150
\\
\hline
$\sqrt{MSE}$ $\widetilde{d}_N^{(WAV)}$& 0.0114 & 0.0147    & 0.0174   & 0.0213   & 0.0258  & 0.0185   & 0.0260  & 0.0446 & 0.0194 & 0.0260
\\
\hline
\end{tabular}
\end{center}
}
\caption{{\small\label{Table1} : Comparison between $\widetilde{d}^{(MIR)}_N$  and other famous semiparametric estimators of $d$ ($\widehat{d}_N^{(IR)}$, $\widetilde{d}_N^{(MS)}$, $\widehat{d}_N^{(ADG)}$  and $\widetilde{d}_N^{(WAV)}$) applied to ARIMA$(1,d,0)$ process (defined by $X_t+\phi X_{t-1}=\varepsilon_t$ for $d=0$ and $(X_t-X_{t-1})+\phi (X_{t-1}-X_{t-2})=\varepsilon_t$ for $d=1$) for  several values of $\phi$ and $N$ and $1000$ independent replications.}
}

\end{table}

\begin{table}[t]
{\footnotesize
\begin{center}
\begin{tabular}{|c|c|c|c|c||c|c|c|c|}
\hline\hline
$N=500$ &$d=-0.2$ &$d=0$ &$d=0.2$ &$d=0.4$  &$d=0.6$ &$d=0.8$ &$d=1$ &$d=1.2$\\
\hline \hline
ARFIMA(0,d,0) &&&&&&&&\\
\hline \hline
$\sqrt{MSE}$ $\widetilde{d}_N^{(MIR)}$  &0.0911    &0.0968    &0.0988    &0.0949    &0.1018    &0.1022    &0.0973    &0.1055
\\
\hline
$\sqrt{MSE}$ $\widehat{d}_N^{(IR)}$   &0.1900    &0.2156    &0.2229    &0.2081    &0.2008    &0.1806    &0.1622    &0.1432
\\
\hline
$\sqrt{MSE}$ $\widetilde{d}_N^{(MS)}$   &0.1405    &0.1441    &0.1432    &0.1523    &0.1681    &0.1765    &0.1703    &0.1643
\\
\hline
$\sqrt{MSE}$ $\widehat{d}_N^{(ADG)}$    &0.0764    &0.0803    &0.0787    &0.0838    &0.0778    &0.0785    &0.0800    &0.0758
\\
\hline
$\sqrt{MSE}$ $\widetilde{d}_N^{(WAV)}$   &0.0716    &0.0795    &0.0849    &0.0865    &0.0808    &0.0848    &0.0701    &0.0707
\\
\hline \hline
ARFIMA(1,d,1) &&&&&&&&\\
\hline \hline
$\sqrt{MSE}$ $\widetilde{d}_N^{(MIR)}$  & 0.1527   & 0.1363   & 0.1315   & 0.1173   & 0.1212   & 0.1099   & 0.1129   & 0.1098
\\
\hline
$\sqrt{MSE}$ $\widehat{d}_N^{(IR)}$   & 0.2255   & 0.2328   & 0.2217   & 0.2205   & 0.2080   & 0.1773   & 0.1592   & 0.1353
\\
\hline
$\sqrt{MSE}$ $\widetilde{d}_N^{(MS)}$   & 0.1393   & 0.1448   & 0.1446   & 0.1521   & 0.1668   & 0.1744   & 0.1688   & 0.1625
\\
\hline
$\sqrt{MSE}$ $\widehat{d}_N^{(ADG)}$ & 0.0939    & 0.0914    & 0.0925    &0.1012   & 0.0933   & 0.0887 & 0.0897   & 0.0872
\\
\hline
$\sqrt{MSE}$ $\widetilde{d}_N^{(WAV)}$    & 0.1728   & 0.1625   & 0.1591   & 0.1424   & 0.1373   & 0.1210   & 0.1026   & 0.0922
\\
\hline \hline
$X^{(d,1,0.1)}$ &&&&&&&&\\
\hline \hline
$\sqrt{MSE}$ $\widetilde{d}_N^{(MIR)}$  & 0.0892  &  0.1003  &  0.1010   & 0.1093&    0.1168 &   0.1126&    0.1158   & 0.1271
\\
\hline
$\sqrt{MSE}$ $\widehat{d}_N^{(IR)}$   & 0.1803   & 0.2081&    0.2186  &  0.2062 &   0.2043  &  0.1840&    0.1700 &   0.1569\\
\hline
$\sqrt{MSE}$ $\widetilde{d}_N^{(MS)}$   & 0.1418  &  0.1438  &  0.1425  &  0.1472 &   0.1538  &  0.1680 &   0.1677  &  0.1697\\
\hline
$\sqrt{MSE}$ $\widehat{d}_N^{(ADG)}$   & 0.0808  &  0.0836 &   0.0804  &  0.0864    &0.0817   & 0.0812&    0.0842 &   0.0817
\\
\hline
$\sqrt{MSE}$ $\widetilde{d}_N^{(WAV)}$   &0.0954  &  0.0871  &  0.0891  &  0.0856  &  0.0772  &  0.0757&    0.0798  &  0.0856\\
\hline \hline
$X^{(d,0,1)}$ &&&&&&&&\\
\hline \hline
$\sqrt{MSE}$ $\widetilde{d}_N^{(MIR)}$  &0.0915    &0.0950    &0.0962   &0.1018    &0.1043    &0.1111   &0.1104    &0.1205
\\
\hline
$\sqrt{MSE}$ $\widehat{d}_N^{(IR)}$   &0.1839    &0.2141    &0.2094   &0.2179    &0.2010    &0.1827   &0.1625  &0.1411\\
\hline
$\sqrt{MSE}$ $\widetilde{d}_N^{(MS)}$   &0.1393   &0.1437   &0.1446   &0.1447   &0.1524    &0.1709    &0.1721    &0.1708\\
\hline
$\sqrt{MSE}$ $\widehat{d}_N^{(ADG)}$   &0.0746   &0.0790  &0.0750   &0.0808    &0.0778   &0.0779    &0.0754    &0.0780
\\
\hline
$\sqrt{MSE}$ $\widetilde{d}_N^{(WAV)}$   &0.0756    &0.0786   &0.0767   &0.0750  &0.0667   &0.0724   &0.0789    &0.0846\\
\hline \hline
$X^{(d,\log)}$ &&&&&&&&\\
\hline \hline
$\sqrt{MSE}$ $\widetilde{d}_N^{(MIR)}$  &0.0836   &0.1064    &0.1089    &0.1161    &0.1138    &0.1197    &0.1252    &0.1380
\\
\hline
$\sqrt{MSE}$ $\widetilde{d}_N^{(IR)}$   &0.1810   &0.2100    &0.2089   &0.2009   &0.1853   &0.1819    &0.1666  &0.1542\\
\hline
$\sqrt{MSE}$ $\widetilde{d}_N^{(MS)}$   &0.1500   &0.1529    &0.1564    &0.1677    &0.1649    &0.1654    &0.1660    &0.1578\\
\hline
$\sqrt{MSE}$ $\widehat{d}_N^{(ADG)}$   &0.0822   &0.0864    &0.0844    &0.0901    &0.0827    &0.0797 &0.0852    &0.0846
\\
\hline
$\sqrt{MSE}$ $\widetilde{d}_N^{(WAV)}$   &0.0974   &0.1087    &0.0996    &0.1068    &0.1007    &0.1031    &0.0967    &0.0829\\
\hline \hline
$X^{(trend)}$ &&&&&&&&\\
\hline \hline
$\sqrt{MSE}$ $\widetilde{d}_N^{(MIR)}$  &0.4684   &0.2922    &0.1633    &0.1051    &0.1027    &0.1176    &0.1176    &0.1279
\\
\hline
$\sqrt{MSE}$ $\widetilde{d}_N^{(IR)}$   &0.2793   &0.2192    &0.2048  &0.2029    &0.1964    &0.1824   &0.1616    &0.1443\\
\hline
$\sqrt{MSE}$ $\widetilde{d}_N^{(MS)}$   &0.9077    &0.6067    &0.3444   &0.2150    &0.2024    &0.1994    &0.1683    &0.1471\\
\hline
$\sqrt{MSE}$ $\widehat{d}_N^{(ADG)}$   &0.5674  &0.3564    &0.1787    &0.1009    &0.0845    &0.0901    &0.0880    &0.0878
\\
\hline
$\sqrt{MSE}$ $\widetilde{d}_N^{(WAV)}$   &0.0961    &0.0908   &0.0886    &0.0913    &0.0917    &0.0907    &0.0804    &0.0896\\
\hline
\end{tabular}\\
~\\
\vspace{3mm}\begin{tabular}{|c|c|c|c|c||c|c|c|c|}
\hline\hline

\hline
\end{tabular}
\end{center}

\caption{{\small \label{Table2}: Comparison between $\widetilde{d}^{(MIR)}_N$  and other famous semiparametric estimators of $d$ ($\widetilde{d}_N^{(IR)}$, $\widetilde{d}_N^{(MS)}$, $\widehat{d}_N^{(ADG)}$  and $\widetilde{d}_N^{(WAV)}$) applied to fractionally integrated processes for $N=500$, several values of $d \in (-0.5,1.25)$ and $1000$ independent replications.}
}
}
\end{table}

\begin{table}[t]
{\footnotesize
\begin{center}

\vspace{3mm}
\begin{tabular}{|c|c|c|c|c||c|c|c|c|}
\hline\hline
$N=5000$ &$d=-0.2$ &$d=0$ &$d=0.2$ &$d=0.4$ &$d=.6$ &$d=0.8$ &$d=1$ &$d=1.2$\\
\hline\hline
ARFIMA(0,d,0) &&&&&&&&\\
\hline \hline
$\sqrt{MSE}$ $\widetilde{d}_N^{(MIR)}$   &0.0391    &0.0318    &0.0329    &0.0346    &0.0363    &0.0381    &0.0399    &0.0513
\\
\hline
$\sqrt{MSE}$ $\widetilde{d}_N^{(IR)}$   &0.0652   &0.0637    &0.0636    &0.0636    &0.0591    &0.0574   &0.0499    &0.0477\\
\hline
$\sqrt{MSE}$ $\widetilde{d}_N^{(MS)}$    &0.0428    &0.0434    &0.0425    &0.0447    &0.0483    &0.0587    &0.0447    &0.1419
\\
\hline
$\sqrt{MSE}$ $\widehat{d}_N^{(ADG)}$    &0.0326    &0.0323    &0.0324    &0.0341    &0.0341    &0.0334    &0.0333    &0.0327
\\
\hline
$\sqrt{MSE}$ $\widetilde{d}_N^{(WAV)}$   &0.0313    &0.0305    &0.0269    &0.0308    &0.0329    &0.0356    &0.0340    &0.0350
\\
\hline\hline
ARFIMA(1,d,1) &&&&&&&& \\
\hline \hline
$\sqrt{MSE}$ $\widetilde{d}_N^{(MIR)}$    &0.0756    &0.0666    &0.0605    &0.0551    &0.0518    &0.0514    &0.0557    &0.0585\\
\hline
$\sqrt{MSE}$ $\widetilde{d}_N^{(IR)}$   &0.1141   &0.0901    &0.0792    &0.0730    &0.0612    &0.0559    &0.0491    &0.0423\\
\hline
$\sqrt{MSE}$ $\widetilde{d}_N^{(MS)}$    &0.0425    &0.0437    &0.0428    &0.0449    &0.0481    &0.0584    &0.0444    &0.1417\\
\hline
$\sqrt{MSE}$ $\widehat{d}_N^{(ADG)}$  &0.0333    &0.0335    &0.0336    &0.0364    &0.0359    &0.0346    &0.0338    &0.0340\\
\hline
$\sqrt{MSE}$ $\widetilde{d}_N^{(WAV)}$    &0.0566    &0.0603    &0.0545    &0.0560    &0.0546    &0.0545    &0.0493    &0.0474\\
\hline\hline
$X^{(d,1,0.1)}$ &&&&&&&&\\
\hline \hline
$\sqrt{MSE}$ $\widetilde{d}_N^{(MIR)}$    & 0.0302  &  0.0401  &  0.0412   & 0.0465 &   0.0427  &  0.0444 &   0.0456  &  0.0490
\\
\hline
$\sqrt{MSE}$ $\widetilde{d}_N^{(IR)}$   & 0.0606  &  0.0678   & 0.0773   & 0.0740  &  0.0652 &   0.0568 &   0.0554 &   0.0472\\
\hline
$\sqrt{MSE}$ $\widetilde{d}_N^{(MS)}$ &0.0429  &  0.0483 &   0.0486  &  0.0502   & 0.0447   & 0.0523  &  0.0458 &   0.1322
\\
\hline
$\sqrt{MSE}$ $\widehat{d}_N^{(ADG)}$   &0.0390  &  0.0410   & 0.0400   & 0.0395   & 0.0357    &0.0378   & 0.0404    &0.0363
\\
\hline
$\sqrt{MSE}$ $\widetilde{d}_N^{(WAV)}$   &0.0363 &   0.0393 &   0.0406 &   0.0375 &   0.0340  &  0.0408&    0.0406    &0.0444
\\
\hline \hline
$X^{(d,0,1)}$ &&&&&&&&\\
\hline \hline
$\sqrt{MSE}$ $\widetilde{d}_N^{(MIR)}$  &0.0330    &0.0297    &0.0314    &0.0320    &0.0319    &0.0315    &0.0339    &0.0395
\\
\hline
$\sqrt{MSE}$ $\widetilde{d}_N^{(IR)}$   &0.0642  &0.0652    &0.0693    &0.0633    &0.0630    &0.0560    &0.0478   &0.0428\\
\hline
$\sqrt{MSE}$ $\widetilde{d}_N^{(MS)}$   &0.0432    &0.0422    &0.0461    &0.0434    &0.0489    &0.0547    &0.0437   &0.1263\\
\hline
$\sqrt{MSE}$ $\widehat{d}_N^{(ADG)}$   &0.0318    &0.0318    &0.0322    &0.0345    &0.0361    &0.0321    &0.0327    &0.0326
\\
\hline
$\sqrt{MSE}$ $\widetilde{d}_N^{(WAV)}$   &0.0271    &0.0298    &0.0248    &0.0319    &0.0348    &0.0331    &0.0384    &0.0402\\
\hline \hline
$X^{(d,\log)}$ &&&&&&&&\\
\hline \hline
$\sqrt{MSE}$ $\widetilde{d}_N^{(MIR)}$  &0.0370 &0.0345    &0.0383    &0.0461  &0.0461   &0.0519    &0.0555  &0.0608
\\
\hline
$\sqrt{MSE}$ $\widetilde{d}_N^{(IR)}$   &0.0627    &0.0683    &0.0676   &0.0659    &0.0587 &0.0528   &0.0490  &0.0501\\
\hline
$\sqrt{MSE}$ $\widetilde{d}_N^{(MS)}$   &0.0582  &0.0615    &0.0632    &0.0640    &0.0636    &0.0584   &0.0552    &0.1300\\
\hline
$\sqrt{MSE}$ $\widehat{d}_N^{(ADG)}$   &0.0417 &0.0427  &0.0414    &0.0411    &0.0401    &0.0405 &0.0415    &0.0403
\\
\hline
$\sqrt{MSE}$ $\widetilde{d}_N^{(WAV)}$   &0.0604   &0.0618    &0.0589    &0.0609    &0.0632   &0.0600    &0.0609  &0.0669\\
\hline \hline
$X^{(trend)}$ &&&&&&&&\\
\hline \hline
$\sqrt{MSE}$ $\widetilde{d}_N^{(MIR)}$  &0.0720   &0.0372   &0.0349    &0.0363    &0.0363  &0.0380 &0.0450    &0.0864
\\
\hline
$\sqrt{MSE}$ $\widetilde{d}_N^{(IR)}$   &0.0677    &0.0639    &0.0690    &0.0655    &0.0602    &0.0545    &0.0485    &0.0506\\
\hline
$\sqrt{MSE}$ $\widetilde{d}_N^{(MS)}$   &0.7760    &0.6067    &0.1480   &0.0675    &0.0680   &0.0750 &0.0443   &0.1512\\
\hline
$\sqrt{MSE}$ $\widehat{d}_N^{(ADG)}$   &0.6019 &0.3613   &0.1502   &0.0555    &0.0387    &0.0377    &0.0369    &0.0364
\\
\hline
$\sqrt{MSE}$ $\widetilde{d}_N^{(WAV)}$   &0.4988   &0.0623   &0.0389    &0.0344    &0.0362    &0.0402    &0.0422    &0.0444\\
\hline
\end{tabular}
\end{center}
}
\caption{{\small \label{Table3}: Comparison between $\widetilde{d}^{(MIR)}_N$  and other famous semiparametric estimators of $d$ ($\widetilde{d}_N^{(IR)}$, $\widetilde{d}_N^{(MS)}$, $\widehat{d}_N^{(ADG)}$  and $\widetilde{d}_N^{(WAV)}$) applied to fractionally integrated processes for $N=5000$, several values of $d \in (-0.5,1.25)$ and $1000$ independent  replications.}
}
\end{table}
\begin{table}[t]
{\footnotesize
\begin{center}

\vspace{3mm}
\begin{tabular}{|c|c|c|c|c||c|c|c|c|}
\hline\hline
$N=50000$ &$d=-0.2$ &$d=0$ &$d=0.2$ &$d=0.4$ &$d=.6$ &$d=0.8$ &$d=1$ &$d=1.2$\\
\hline\hline
ARFIMA(0,d,0) &&&&&&&&\\
\hline \hline
$\sqrt{MSE}$ $\widetilde{d}_N^{(MIR)}$   &0.0201   &0.0081  &0.0132    &0.0141   &0.0139  &0.0150 &0.0128   &0.0294
\\
\hline
$\sqrt{MSE}$ $\widetilde{d}_N^{(IR)}$   &0.0248  &0.0219    &0.0218    &0.0216    &0.0191    &0.0203   &0.0138    &0.0144\\
\hline
$\sqrt{MSE}$ $\widetilde{d}_N^{(MS)}$    &0.0151   &0.0161   &0.0150    &0.0127    &0.0178    &0.0217    &0.0127   &0.1595
\\
\hline
$\sqrt{MSE}$ $\widehat{d}_N^{(ADG)}$    &0.0143    &0.0150    &0.0160    &0.0134    &0.0159   &0.0133    &0.0139   &0.0149
\\
\hline
$\sqrt{MSE}$ $\widetilde{d}_N^{(WAV)}$   &0.0102    &0.0079    &0.0086   &0.0102    &0.0107    &0.0196    &0.0183  &0.0253
\\
\hline\hline
ARFIMA(1,d,1) &&&&&&&& \\
\hline \hline
$\sqrt{MSE}$ $\widetilde{d}_N^{(MIR)}$    &0.0440    &0.0278    &0.0247    &0.0232   &0.0185    &0.0233    &0.0198   &0.0326\\
\hline
$\sqrt{MSE}$ $\widetilde{d}_N^{(IR)}$   &0.0906   &0.0658    &0.0479    &0.0355    &0.0298    &0.0230   &0.0194    &0.0163\\
\hline
$\sqrt{MSE}$ $\widetilde{d}_N^{(MS)}$    &0.0146    &0.0125    &0.0142    &0.0163    &0.0179    &0.0257    &0.0141   &0.1564\\
\hline
$\sqrt{MSE}$ $\widehat{d}_N^{(ADG)}$  &0.0160   &0.0137   &0.0144    &0.0160    &0.0156    &0.0158    &0.0154    &0.0138\\
\hline
$\sqrt{MSE}$ $\widetilde{d}_N^{(WAV)}$    &0.0233   &0.0252    &0.0268    &0.0210   &0.0179  &0.0257    &0.0254    &0.0319\\
\hline\hline
$X^{(d,1,0.1)}$ &&&&&&&&\\
\hline \hline
$\sqrt{MSE}$ $\widetilde{d}_N^{(MIR)}$    &0.0093  &  0.0243    &0.0268  &  0.0273  &  0.0280   & 0.0265    &0.0249 &   0.0224
\\
\hline
$\sqrt{MSE}$ $\widetilde{d}_N^{(IR)}$   &0.0182  &  0.0330 &   0.0349 &   0.0342  &  0.0335   & 0.0316   & 0.0258  &  0.0267\\
\hline
$\sqrt{MSE}$ $\widetilde{d}_N^{(MS)}$ &0.0244  &  0.0293   & 0.0267  &  0.0276  &  0.0251    &0.0216&    0.0216 &   0.1375
\\
\hline
$\sqrt{MSE}$ $\widehat{d}_N^{(ADG)}$   & 0.0243  &  0.0283   & 0.0257 &   0.0265 &   0.0230   & 0.0248  &  0.0244  &  0.0253
\\
\hline
$\sqrt{MSE}$ $\widetilde{d}_N^{(WAV)}$   &0.0232  &  0.0290  &  0.0273   & 0.0397  &  0.0290  &  0.0281 &   0.0228 &   0.0318
\\
\hline \hline
$X^{(d,0,1)}$ &&&&&&&&\\
\hline \hline
$\sqrt{MSE}$ $\widetilde{d}_N^{(MIR)}$  &0.0181    &0.089    &0.0107   &0.0110    &0.0108    &0.0125    &0.0115    &0.0121
\\
\hline
$\sqrt{MSE}$ $\widetilde{d}_N^{(IR)}$   &0.0273   &0.0205    &0.0236    &0.0215    &0.0221    &0.0159   &0.0147    &0.0131\\
\hline
$\sqrt{MSE}$ $\widetilde{d}_N^{(MS)}$   &0.0140   &0.0154    &0.0151    &0.0166    &0.0167    &0.0227    &0.0159    &0.1337\\
\hline
$\sqrt{MSE}$ $\widehat{d}_N^{(ADG)}$   &0.0148    &0.0165    &0.0167    &0.0177    &0.0146    &0.0145    &0.0161   &0.0154
\\
\hline
$\sqrt{MSE}$ $\widetilde{d}_N^{(WAV)}$   &0.0099    &0.0167    &0.0135    &0.0156    &0.0189  &0.0148   &0.0283    &0.0268\\
\hline \hline
$X^{(d,\log)}$ &&&&&&&&\\
\hline \hline
$\sqrt{MSE}$ $\widetilde{d}_N^{(MIR)}$  &0.0193   &0.0240    &0.0287    &0.0312   &0.0382    &0.0390    &0.0419    &0.0472
\\
\hline
$\sqrt{MSE}$ $\widetilde{d}_N^{(IR)}$   &0.0300 &0.0256    &0.0282    &0.0294    &0.0210   &0.0191  &0.0244    &0.0305\\
\hline
$\sqrt{MSE}$ $\widetilde{d}_N^{(MS)}$   &0.0463  &0.0498    &0.0480    &0.0504    &0.0478    &0.0408   &0.0418    &0.1480\\
\hline
$\sqrt{MSE}$ $\widehat{d}_N^{(ADG)}$   &0.0456 &0.0475  &0.0464   &0.0464    &0.0438    &0.0456 &0.0468    &0.0453
\\
\hline
$\sqrt{MSE}$ $\widetilde{d}_N^{(WAV)}$   &0.0529   &0.0515    &0.0509    &0.0524    &0.0465   &0.0468    &0.0544  &0.0498\\
\hline \hline
$X^{(trend)}$ &&&&&&&&\\
\hline \hline
$\sqrt{MSE}$ $\widetilde{d}_N^{(MIR)}$  &0.0271   &0.0097   &0.0127    &0.0130    &0.0132  &0.0132 &0.0126    &0.0562
\\
\hline
$\sqrt{MSE}$ $\widetilde{d}_N^{(IR)}$   &0.0282   &0.0228   &0.0226    &0.0211    &0.0199    &0.0160   &0.0165    &0.0194\\
\hline
$\sqrt{MSE}$ $\widetilde{d}_N^{(MS)}$   &0.9840   &0.6253    &0.1134  &0.0194    &0.0224   &0.0395 &0.0117   &0.1655\\
\hline
$\sqrt{MSE}$ $\widehat{d}_N^{(ADG)}$   &0.6190 &0.3616   &0.1356   &0.0209   &0.0158    &0.0156    &0.0153    &0.0155
\\
\hline
$\sqrt{MSE}$ $\widetilde{d}_N^{(WAV)}$   &1.0023   &0.5575  &0.0386    &0.0182    &0.0181    &0.0253    &0.0474    &0.0275\\
\hline
\end{tabular}
\end{center}
}
\caption{{\small \label{Table4}: Comparison between $\widetilde{d}^{(MIR)}_N$  and other famous semiparametric estimators of $d$ ($\widetilde{d}_N^{(IR)}$, $\widetilde{d}_N^{(MS)}$, $\widehat{d}_N^{(ADG)}$  and $\widetilde{d}_N^{(WAV)}$) applied to fractionally integrated processes for $N=50000$, several values of $d \in (-0.5,1.25)$ and $1000$ independent replications.}
}
\end{table}
~\\
\underline{{\bf Conclusions of simulations:}} Even if the estimator  $\widehat{d}_N^{(ADG)}$ often provides the more accurate estimation of $d$ for stationary processes, it is not more accurate anymore than $\widehat{d}_N^{(MIR)}$ in case of trended time series. Moreover since this is not a data-driven estimator, with a bandwidth $m$ fixed to be $N^{0.65}$, it is not a consistent estimator when $\beta$ is small enough: this is such the case for $X^{(d,\log)}$ where we observe that the MSE is globally larger for $N=50000$ than for $N=5000$. The estimator $\widehat{d}_N^{(MIR)}$ is a very good trade-off with always one of the  smallest $\sqrt{MSE}$ among the $5$ semiparametric estimators and almost never bad results (except perhaps for $X^{(trend)}$, $N=500$ and $d<0.5$). Moreover, the larger $N$ the more efficient $\widehat{d}_N^{(MIR)}$ with respect to the other estimators. 
Note also that the use of a data-driven multidimensional version of $IR$ statistics ({\it i.e.} the estimator $\widehat{d}_N^{(MIR)}$) considerably improves the quality of the estimation with respect to the original estimator based on unidimensional IR statistics (the estimator $\widehat{d}_N^{(IR)}$). Finally the other data-driven estimators $ d^{(MS)}$ and $d^{(WAV)}$ provide correct results but are often less efficient than $\widehat{d}_N^{(MIR)}$.

\subsubsection{Comparison of MIR tests $\widetilde S_N$ and $\widetilde T_N$  with other stationarity or non-stationarity tests}
Monte-Carlo experiments were done for evaluating the performances of new tests $\widetilde S_N$ and $\widetilde T_N$ and for comparing them to most famous stationarity tests (KPSS and V/S) or non-stationarity (ADF and PP) tests (see more details on these tests in the previous section). \\
As it is suggested for the corresponding  R-software commands (see also Banerjee {\it et al.}, 1993), we chose the following trimming parameters for the classical tests:
\begin{itemize}
\item $k=\Big [\frac 3 {13}\, \sqrt n \Big ]$ for KPSS test;
\item $k=\sqrt N$ for V/S test;
\item $k=\Big [(N-1)^{1/3}\Big ]$ for ADF test;
\item  $k=\Big [4\, \big (\frac N {100} \big )^{1/4}\Big ]$ for PP test;
\end{itemize}
The results of these simulations  with a type I error classically chosen to $0.05$ are provided in Tables \ref{Table5}, \ref{Table6}, \ref{Table7} and \ref{Table8}.
\\

\begin{table}
{\footnotesize
\begin{center}
\begin{tabular}{|c|c|c|c|c|c||c|c|c|c|c|}
\hline\hline
 $N=500$  &$d=0$ & $d=0$ &  $d=0$ & $d=0$ & $d=0$  &$d=1$ & $d=1$  & $d=1$ & $d=1$  & $d=1$ \\
ARIMA$(1,d,0)$ &$\phi$=-0.1 &$\phi$=-0.3 &$\phi$=-0.5 &$\phi$=-0.7 &$\phi$=-0.9 &$\phi$=-0.1 &$\phi$=-0.3  &$\phi$=-0.5  &$\phi$=-0.7 &$\phi$=-0.9 ~ \\
\hline \hline
$\widetilde S_N$: Rejected $H_{0}$   &0  &0   &0    &0   &0.508 &  0.992   &0.992 & 0.993 & 0.995 &  1.000\\
$KPSS$: Rejected $H_{0}$   &0.058    &0.091    &0.125  &0.228  &0.679    &0.998 & 0.998 & 0.999 & 1.000 &  1.000
\\
$V/S$ : Rejected $H_{0}$  &0.057   & 0.071 &   0.105 &   0.207  &  0.680   & 0.997  &  0.998  &  0.999  &  1.000  &  1.000
\\
\hline
\hline
$\widetilde T_N$: Rejected $H'_{0}$  & 0.998 & 0.995  &0.990    &0.845    &0.074         &0         &0         &0 & 0 &0
\\
ADF: Rejected $H'_{0}$    &1.000    &1.000    &1.000  & 1.000 & 1.000  &0.048    &0.043    &0.043 & 0.041 & 0.049 \\

PP : Rejected $H'_{0}$    &1.000    &1.000    &1.000  &1.000    &1.000   &0.040   &0.032    &0.017 & 0.012 & 0\\

\hline
\hline
\end{tabular}\\
~\\
\vspace{3mm}
\begin{tabular}{|c|c|c|c|c|c||c|c|c|c|c|}
\hline\hline
 $N=5000$  &$d=0$ & $d=0$ &  $d=0$ & $d=0$ & $d=0$  &$d=1$ & $d=1$  & $d=1$ & $d=1$  & $d=1$ \\
ARIMA$(1,d,0)$ &$\phi$=-0.1 &$\phi$=-0.3 &$\phi$=-0.5 &$\phi$=-0.7 &$\phi$=-0.9 &$\phi$=-0.1 &$\phi$=-0.3  &$\phi$=-0.5  &$\phi$=-0.7 &$\phi$=-0.9 ~ \\
\hline \hline
$\widetilde S_N$: Rejected $H_{0}$   &0   &0   &0         &0         &0.118         &1.000 & 1.000 & 1.000 & 1.000 & 1.000 \\
$KPSS$: Rejected $H_{0}$ &0. 044  &0.045    &0.084    & 0.078 &   0.306  &1.000    & 1.000        &1.000 & 1.000 & 1.000
\\
$V/S$ : Rejected $H_{0}$  &0.053  &  0.053 &   0.063  &  0.088  &  0.295   & 1.000 &   1.000 &   1.000 &   1.000&    1.000 \\
\hline
\hline
$\widetilde T_N$: Rejected $H'_{0}$ & 1.000    &1.000 & 1.000    &1.000    &0.870        & 0        &0         &0 & 0 & 0
\\
ADF: Rejected $H'_{0}$  &1.000    &1.000    &1.000 &1.000    &1.000    &0.034    &0.051   &0.042  & 0.044 & 0.068 \\

PP : Rejected $H'_{0}$ &1.000    &1.000  &1.000    &1.000    &1.000    &0.029   &0.058    &0.031 & 0.024 & 0.008
\\
\hline
\hline
\end{tabular}
~\\
\vspace{1cm}
\begin{tabular}{|c|c|c|c|c|c||c|c|c|c|c|}
\hline\hline
 $N=500$  &$d=0$ & $d=0$ &  $d=0$ & $d=0$ & $d=0$  &$d=1$ & $d=1$  & $d=1$ & $d=1$  & $d=1$ \\
ARIMA$(1,d,0)$ &$\phi$=0.1~ &$\phi$=0.3~ &$\phi$=0.5~ &$\phi$=0.7~ &$\phi$=0.9 ~&$\phi$=0.1~ &$\phi$=0.3~  &$\phi$=0.5~  &$\phi$=0.7 ~&$\phi$=0.9 ~ \\
\hline \hline
$\widetilde S_N$: Rejected $H_{0}$   &0  &0   &0    &0   &0 &  0.990   &0.995 & 0.994 & 0.995 &  0.995\\
$KPSS$: Rejected $H_{0}$   &0.040   &0.029    &0.025  &0.010  &0.007   &0.998 & 0.998 & 0.997 & 0.998 & 0.999
\\
$V/S$ : Rejected $H_{0}$  & 0.043  &  0.030  &  0.018 &   0.012   & 0.006    &1.000    &0.999  &  1.000&    1.000 &   0.999
\\
\hline
\hline
$\widetilde T_N$: Rejected $H'_{0}$  & 0.998 & 1.000  &0.999    &1.000    &1.000         &0         &0         &0 & 0 &0
\\
ADF: Rejected $H'_{0}$    &1.000    &1.000    &1.000  & 1.000 & 1.000  &0.040    &0.048    &0.038 & 0.040 & 0.055 \\

PP : Rejected $H'_{0}$    &1.000    &1.000    &1.000  &1.000    &1.000   &0.041   &0.074   &0.108 & 0.226& 0.534\\

\hline
\hline
\end{tabular}\\
~\\
\vspace{3mm}
\begin{tabular}{|c|c|c|c|c|c||c|c|c|c|c|}
\hline\hline
 $N=5000$  &$d=0$ & $d=0$ &  $d=0$ & $d=0$ & $d=0$  &$d=1$ & $d=1$  & $d=1$ & $d=1$  & $d=1$ \\
ARIMA$(1,d,0)$ &$\phi$=0.1~ &$\phi$=0.3~ &$\phi$=0.5~ &$\phi$=0.7~ &$\phi$=0.9 ~&$\phi$=0.1~ &$\phi$=0.3~  &$\phi$=0.5~  &$\phi$=0.7 ~&$\phi$=0.9 ~ \\
\hline \hline
$\widetilde S_N$: Rejected $H_{0}$   & 0   &  0 &    0&     0 &    0  &  1.000    &1.000 &   1.000  &  1.000 &   1.000 \\
$KPSS$: Rejected $H_{0}$ &0.087 &   0.044 &   0.041  &  0.016  &  0.008   & 1.000 &   1.000  &  1.000 &   1.000 &   1.000
\\
$V/S$ : Rejected $H_{0}$  & 0.068 &   0.035   & 0.044  &  0.019 &   0.003   & 1.000    &1.000 &   1.000  &  1.000 &   1.000\\
\hline
\hline
$\widetilde T_N$: Rejected $H'_{0}$ & 1.000    &1.000 &   1.000  &  1.000 &   1.000&     0   &  0  &   0   &  0  &   0
\\
ADF: Rejected $H'_{0}$  &  1.000  &  1.000   & 1.000 &   1.000 &   1.000 &   0.025 &   0.049   & 0.030   & 0.074    &0.041 \\

PP : Rejected $H'_{0}$ & 1.000  &  1.000   & 1.000  &  1.000  &  1.000   & 0.033  &  0.057  &  0.052  &  0.144   & 0.352
\\
\hline
\hline
\end{tabular}
\end{center}
}
\caption{{\small \label{Table5} Comparisons of stationarity and non-stationarity tests from $1000$ independent Monte Carlo experiment replications of ARIMA$(1,d,0)$ processes (defined by $X_t+\phi X_{t-1}=\varepsilon_t$ for $d=0$ and $(X_t-X_{t-1})+\phi (X_{t-1}-X_{t-2})=\varepsilon_t$ for $d=1$) for  several values of $\phi$ and $N$. The accuracy of tests is measured by the rejection  probabilities.}
}
\end{table}

\begin{table}
{\footnotesize
\begin{center}
\begin{tabular}{|c|c|c|c|c||c|c|c|c|}
\hline\hline
$N=500$ &&&&&&&&\\
ARFIMA$(0,d,0)$ &$d=-0.2$ &$d=0$ &$d=0.2$ &$d=0.4$ &$d=0.6$ &$d=0.8$ &$d=1$ &$d=1.2$ \\
\hline \hline
$\widetilde S_N$: Rejected  $H_{0}$  & 0      &   0   &      0  &  0.003 &   0.276  &  0.917    &0.998   & 0.999   \\
$KPSS$: Rejected  $H_{0}$ &0 &   0.059 &   0.395   & 0.771  &  0.946    &0.989   & 0.999   & 0.999  \\
$V/S$ : Rejected  $H_{0}$ & 0   & 0.052  &  0.446 &   0.847 &   0.970 &   0.993  &  0.998    &1.000 \\
\hline\hline
$\widetilde T_N$: Rejected $H'_{0}$  & 1.000    &1.000  &  0.965 &   0.421  &  0.017   &      0    &     0   &      0\\
ADF: Rejected $H'_{0}$  & 1.000   & 1.000  &  1.000   & 0.977  &  0.615  &  0.233 &   0.065  &  0.005   \\
PP : Rejected $H'_{0}$ &1.000 &   1.000   & 1.000  &  1.000  &  0.919  &  0.447 &   0.065 &   0.002  \\
\hline
\hline
\end{tabular}\\
~\\
\vspace{3mm}
\begin{tabular}{|c|c|c|c|c||c|c|c|c|}
\hline\hline
$N=5000$ &&&&&&&&\\
ARFIMA$(0,d,0)$ &$d=-0.2$ &$d=0$ &$d=0.2$ &$d=0.4$ &$d=0.6$ &$d=0.8$ &$d=1$ &$d=1.2$ \\
\hline \hline
$\widetilde S_N$: Rejected $H_{0}$  & 0     &    0   &      0   &      0   & 0.912   & 1.000 &   1.000 &   1.000 \\
$KPSS$: Rejected  $H_{0}$ & 0  &  0.042   & 0.674   & 0.996    &1.000  &  1.000   & 1.000   & 1.000 \\
$V/S$ : Rejected  $H_{0}$ &0  &  0.038  &  0.694  &  0.992  &  1.000   & 1.000  &  1.000  &  1.000\\
\hline\hline
$\widetilde T_N$: Rejected $H'_{0}$  & 1.000   & 1.000    &1.000  &  0.946       &  0      &   0    &     0       &  0\\
ADF: Rejected $H'_{0}$  &1.000  &  1.000  &  1.000  &  1.000  &  0.946   & 0.448 &   0.050    &0.004   \\
PP : Rejected $H'_{0}$  &1.000  &  1.000  &  1.000   & 1.000    &1.000  &  0.705 &   0.042  &   0 \\
\hline
\hline
\end{tabular}
\end{center}
}

\caption{{\small\label{Table6}Comparisons of stationarity and non-stationarity tests from $1000$ independent Monte Carlo experiment replications of ARFIMA$(0,d,0)$ processes for  several values of $d$ and $N$. The accuracy of tests is measured by the rejection probabilities.}
}
\end{table}

\begin{table}
{\footnotesize
\begin{center}
\begin{tabular}{|c|c|c|c|c||c|c|c|c|}
\hline\hline
$N=500$ &&&&&&&&\\
ARFIMA$(1,d,1)$ &$d=-0.2$ &$d=0$ &$d=0.2$ &$d=0.4$ &$d=0.6$ &$d=0.8$ &$d=1$ &$d=1.2$ \\
$\phi=-0.3$ ; $\theta=0.7$ &&&&&&&&\\
\hline \hline
$\widetilde S_N$: Rejected  $H_{0}$  & 0     &    0   &      0   & 0.015   & 0.442   & 0.898 &   0.987 &   0.999  \\
$KPSS$: Rejected  $H_{0}$ & 0  &  0.079  &  0.454  &  0.836  &  0.959   & 0.995  &  0.997  &  0.999 \\
$V/S$ : Rejected $H_{0}$ & 0.001  &  0.077   & 0.481  &  0.876 &   0.974  &  0.993 &   1.000 &   1.000\\
\hline\hline
$\widetilde T_{N}$: Rejected $H'_{0}$ & 0.999   & 0.990   & 0.823  &  0.212  &  0.009     &    0    &     0     &    0
\\
ADF: Rejected $H'_{0}$ &1.000 &   1.000 &   1.000 &   0.961    &0.623    &0.230   & 0.056   & 0.010 \\
PP : Rejected $H'_{0}$ & 1.000  &  1.000  &  1.000  &  0.999   & 0.781  &  0.270 &   0.036     &    0 \\
\hline
\hline
\end{tabular}\\
~\\
\vspace{3mm}
\begin{tabular}{|c|c|c|c|c||c|c|c|c|}
\hline\hline
$N=5000$ &&&&&&&&\\
ARFIMA$(1,d,1)$ &$d=-0.2$ &$d=0$ &$d=0.2$ &$d=0.4$ &$d=0.6$ &$d=0.8$ &$d=1$ &$d=1.2$ \\
$\phi=-0.3$ ; $\theta=0.7$ &&&&&&&&\\
\hline \hline
$\widetilde S_N$: Rejected $H_{0}$  &0     &    0   &      0  &  0.004   & 0.846   & 1.000   & 1.000    &1.000 \\
$KPSS$: Rejected  $H_{0}$ &0  &  0.060   & 0.689    &0.963  &  0.996  &  1.000 &   1.000  &  1.000  \\
$V/S$ : Rejected $H_{0}$  & 0  &  0.060 &   0.697 &   0.989  &  1.000 &   1.000 &   1.000  &  1.000\\
\hline\hline
$\widetilde T_{N}$: Rejected $H'_{0}$ &1.000   & 1.000  &  1.000   & 0.700  &  0.008  &       0   &      0    &     0
\\
ADF: Rejected $H'_{0}$  &1.000   & 1.000 &   1.000  &  1.000   & 0.951 &   0.371  &  0.052  &  0.004  \\
PP : Rejected $H'_{0}$ & 1.000 &   1.000  &  1.000   & 1.000    &0.996  &  0.584&    0.038   &  0  \\
\hline
\hline
\end{tabular}
\end{center}
}
\caption{{\small \label{Table7} Comparisons of stationarity and non-stationarity tests from $1000$ independent Monte Carlo experiment replications of ARFIMA$(1,d,1)$ processes with $\phi=-0.3$ and $\theta=0.7$  for  several values of $d$ and $N$. The accuracy of tests is measured by the rejection probabilities.}
}
\end{table}

\begin{table}
{\footnotesize
\begin{center}
\begin{tabular}{|c|c|c|c|c||c|c|c|c|}
\hline\hline
$N=500$ &&&&&&&&\\
$X^{(d,0,1)}$ &$d=-0.2$ &$d=0$ &$d=0.2$ &$d=0.4$ &$d=0.6$ &$d=0.8$ &$d=1$ &$d=1.2$ \\
\hline \hline
$\widetilde S_N$: Rejected  $H_{0}$  &0      &   0     &    0  &  0.001  &  0.294  &  0.883  &  0.990   & 0.999 \\
$KPSS$: Rejected  $H_{0}$ &0  &  0.052  &  0.433   & 0.801   & 0.939  &  0.988   & 0.999   & 1.000  \\
$V/S$ : Rejected  $H_{0}$ & 0  &  0.035  &  0.464  &  0.844  &  0.963 &   0.995 &   0.999  &  1.000 \\
\hline\hline
$\widetilde T_{N}$: Rejected $H'_{0}$ & 1.000  &  1.000  &  0.953  &  0.405  &  0.022    &     0   &      0   &      0\\
ADF: Rejected $H'_{0}$    &  1.000   & 1.000    &1.000   & 1.000 &   0.976   & 0.561   & 0.188   & 0.073  \\
PP : Rejected $H'_{0}$    &  1.000   & 1.000  &  1.000   & 1.000   & 1.000   & 0.803  &  0.184 &   0.049 \\
\hline
\hline
\end{tabular}\\
~\\
\vspace{3mm}
\begin{tabular}{|c|c|c|c|c||c|c|c|c|}
\hline\hline
$N=5000$ &&&&&&&&\\
$X^{(d,0,1)}$ &$d=-0.2$ &$d=0$ &$d=0.2$ &$d=0.4$ &$d=0.6$ &$d=0.8$ &$d=1$ &$d=1.2$ \\
\hline \hline
$\widetilde S_N$: Rejected  $H_{0}$   & 0  &   0 &    0  &   0  &0.933  &1.000 & 1.000  &  1.000  \\
$KPSS$: Rejected $H_{0}$  & 0   & 0.082   & 0.689   & 0.970   & 1.000   & 1.000   & 1.000  &  1.000 \\
$V/S$ : Rejected  $H_{0}$  & 0  &  0.075    &0.723   & 0.970  &  0.996  &  1.000  &  1.000   & 1.000 \\
\hline\hline
$\widetilde T_N$: Rejected $H'_{0}$  & 1.000 &   1.000 &   1.000   & 0.940    &     0    &     0    &     0      &   0
\\
ADF: Rejected $H'_{0}$   &1.000  &  1.000  &  1.000  &  1.000  &  1.000   & 0.753  &  0.124   & 0.086  \\
PP : Rejected $H'_{0}$  &1.000  &  1.000   & 1.000 &   1.000   & 1.000   & 0.918   & 0.109   & 0.139  \\
\hline
\hline
\end{tabular}
\end{center}
}
\caption{{ \small \label{Table8}
Comparisons of stationarity and non-stationarity tests from $1000$ independent Monte Carlo experiment replications of $X^{(d,0,1)}$ processes for  several values of $d$ and $N$.  The accuracy of tests is measured by the rejection  probabilities.}
}
\end{table}

~\\
\underline{{\bf Conclusions of simulations:}} As it is well known, from their constructions, KPSS , V/S, ADF and PP tests should asymptotically decide the stationarity hypothesis when $d=0$, and the non-stationarity hypothesis when $d>0$. It was exactly what we observe in these simulations. For ARIMA$(p,d,0)$ processes with $d=0$ or $d=1$ ({\it i.e.} AR$(1)$ process when $d=0$), ADF and PP  tests are more efficient tests than our adaptive MIR tests when $N=500$. However, note that all stationarity tests do not control the size for $\phi=-0.9$. But when $N=5000$ the tests computed from $\widetilde d_N^{(MIR)}$ provide comparable and convincing results. Note also that KPSS and V/S provide reasonable results but less efficient than the other tests.
In case of processes with $d\in (0,1)$, the tests computed from $\widetilde d_N^{(MIR)}$ obtain clearly better performances than classical non-stationarity tests ADF or PP which accept the non-stationarity assumption $H_0'$ even if the processes are stationary when $0<d<0.5$ for instance. The case of the V/S test is different since this test is built for distinguishing between short and long-memory processes. However, as it was already established in Giraitis {\it et al.} (2003), V/S test is slightly more accurate than KPSS for testing the stationarity. Note also that a renormalized version of this test has been defined in Giraitis {\it et al.} (2006) for taking into account the value of $d$.

\subsubsection{Comparison of MIR Fractional Unit Roots test $\widetilde F_N$ and Dolado {\it et al.}  and Lobato and Velasco Fractional Unit Roots tests}
Monte-Carlo experiments were also done for evaluating the performances of new Fractional Unit Root test $\widetilde F_N$ and for comparing it to the Fractional Unit Roots tests defined in Dolado {\it et al. } (2002) and in Lobato and Velasco (2007).
\begin{enumerate}
\item The student-type test statistic defined in Dolado, Gonzalo and Mayoral (2002) is such as:
\begin{equation}\label{TDGM}
\widehat T_{DGM}= \frac {\sum_{t=2}^N(X_t-X_{t-1})\Delta^{\widehat d_1} X_{t-1}}{ \Big ( \sum_{ t=2}^N \big (\Delta^{\widehat d_1} X_{t-1}\big )^2 \times \frac 1 N \sum_{t=2}^N \big (X_t-X_{t-1} - \widehat \phi \Delta^{\widehat d_1} X_{t-1}\big )^2   \Big )^{1/2} },
\end{equation}
with $\widehat \phi=\frac { \sum_{t=2}^N(X_t-X_{t-1})\Delta^{\widehat d_1} X_{t-1}} {   \sum_{ t=2}^N \big (\Delta^{\widehat d_1} X_{t-1}\big )^2} $ and $\widehat d_1$ is the minimum between an ordinary least square estimator of $d_1$ and $1-c$ with $c>0$ small enough (typically $c=0.02$). This is an extension in a fractional framework of the Dickey-Fuller test.\\
\item The efficient Wald test statistic defined in Lobato and Velasco (2007) is based on a two-step student test of a regression coefficient, {\it i.e.} $\widehat T_{LV}$ is the student test of the $(z_t)_t$ coefficient for the regression of $X_t-X_{t-1}$ onto variables $z_t-\widehat \alpha_1 z_{t-1}-\cdots-\widehat \alpha_p z_{t-p}$, $X_{t-1}-X_{t-2}, \cdots , X_{t-p}-X_{t-p-1}$ for $t=p+1,\cdots, n$, where $z_t$ is defined as
\begin{equation}\label{TLV}
z_t= \frac {\Delta^{\widehat d} X_{t}-(X_t-X_{t-1})}{1-\widehat d},
\end{equation}
and where $(\widehat \alpha_1,\cdots,\widehat \alpha_p)$ are obtained as a minimizer of $\sum_{k=p}^n \big (\Delta^{\widehat d} X_{t}-\alpha_1\Delta^{\widehat d} X_{t-1}-\cdots -\alpha_p\Delta^{\widehat d} X_{t-p}\big )^2$ and $\widehat d$ is a semi-parametric local Whittle type estimator of $d$. Note that $\widehat T_{LV}$ is depending on $p$ and in the sequel we will chose $p=1$ and $p=10$, defining respectively $\widehat T_{LV1}$ and $\widehat T_{LV10}$.
\end{enumerate}
We applied the fractional unit roots tests $\widetilde F_N$,  $\widehat T_{LV1}$ and $\widehat T_{LV10}$ to several fractional processes and displayed the results in Table \ref{Table9}. As it is a test specially devoted to FARIMA$(0,d,0)$ processes, we only applied the fractional unit roots test $\widehat T_{DGM}$ to those processes. Finally, note that we also consider the additional sample size $N=200$ to $N=500$ and $N=5000$ used in other simulations because this could help to better evaluate the power of the several tests (since for $N=500$ and $N=5000$ the test powers are often $1$).

\begin{table}
{\footnotesize
\begin{center}
\begin{tabular}{|c|c|c|c|c|c|c|}
\hline\hline
$N=200$ &&&&&&\\
ARFIMA$(0,d,0)$ &$d=0.5$ &$d=0.6$ &$d=0.7$ &$d=0.8$ &$d=0.9$ &$d=1$ \\
\hline \hline
$\widetilde F_N$: Rejected $H^{FUR}_{0}$  & 0.995 &   0.960&    0.819  &  0.577   & 0.292 &   0.132  \\
$\widehat T_{DGM}$: Rejected $H^{FUR}_{0}$  & 0.991   & 0.982 &   0.956 &   0.734 &   0.277  &  0.058
\\
$\widehat T_{LV1}$: Rejected $H^{FUR}_{0}$  &  0.999 &  0.998 &  0.932&    0.604&    0.185&   0.059
\\
$\widehat T_{LV10}$: Rejected $H^{FUR}_{0}$  &  0.414 &   0.276&   0.206&    0.128&    0.069&   0.044
\\
\hline
ARFIMA$(1,d,0)$ &&&&&& \\
\hline \hline
$\widetilde F_N$: Rejected $H^{FUR}_{0}$  & 0.998&   0.975 &   0.881&    0.653&  0.388&    0.136\\
$\widehat T_{LV1}$: Rejected $H^{FUR}_{0}$  &  1.000 &   1.000&  1.000&  0.992 &    0.851& 0.056
\\
$\widehat T_{LV10}$: Rejected $H^{FUR}_{0}$  &  0.433 &   0.308&   0.205&    0.118&    0.062&  0.043
\\
\hline
ARFIMA$(1,d,1)$ &&&&&& \\
\hline \hline
$\widetilde F_N$: Rejected $H^{FUR}_{0}$  & 0.961&   0.870&    0.654&    0.416&  0.183&    0.076\\
$\widehat T_{LV1}$: Rejected $H^{FUR}_{0}$  &  0.996 &   0.942&   0.520&    0.087&    0.099&  0.571
\\
$\widehat T_{LV10}$: Rejected $H^{FUR}_{0}$  &  0.353 &   0.277&   0.144&    0.095&    0.067& 0.044
\\
\hline $X^{(d,0,1)}$ &&&&&&\\
\hline \hline
$\widetilde F_N$: Rejected $H^{FUR}_{0}$  & 0.993&    0.956&   0.825&    0.569&    0.318&    0.125\\
$\widehat T_{LV1}$: Rejected $H^{FUR}_{0}$  &  1.000 & 0.999& 0.965  &  0.716 &    0.358&   0.106
\\
$\widehat T_{LV10}$: Rejected $H^{FUR}_{0}$  &  0.682 &   0.504 &   0.294&    0.173&    0.110&   0.057
\\
\hline
\hline
\end{tabular}\\
~\\
\vspace{3mm}
\begin{tabular}{|c|c|c|c|c|c|c|}
\hline\hline
$N=500$ &&&&&&\\
ARFIMA$(0,d,0)$ &$d=0.5$ &$d=0.6$ &$d=0.7$ &$d=0.8$ &$d=0.9$ &$d=1$ \\
\hline \hline
$\widetilde F_N$: Rejected $H^{FUR}_{0}$  & 0.998 &   0.991&    0.968  &  0.816   & 0.416 &   0.101  \\
$\widehat T_{DGM}$: Rejected $H^{FUR}_{0}$  & 0.998   & 0.998 &   0.999 &   0.984 &   0.607  &  0.049
\\
$\widehat T_{LV1}$: Rejected $H^{FUR}_{0}$  &  1.000 &   1.000 &   1.000&    0.960&    0.441&   0.052
\\
$\widehat T_{LV10}$: Rejected $H^{FUR}_{0}$  &  0.912 &   0.783&   0.527&    0.243&    0.095&   0.048
\\
\hline
ARFIMA$(1,d,0)$ &&&&&& \\
\hline \hline
$\widetilde F_N$: Rejected $H^{FUR}_{0}$  & 0.998&   0.994&    0.952&    0.814&  0.510&    0.145\\
$\widehat T_{LV1}$: Rejected $H^{FUR}_{0}$  &  1.000 &   1.000&  1.000&  1.000 &    0.669& 0.053
\\
$\widehat T_{LV10}$: Rejected $H^{FUR}_{0}$  &  0.900 &   0.787&   0.509&    0.270&    0.115&  0.048
\\
\hline
ARFIMA$(1,d,1)$ &&&&&& \\
\hline \hline
$\widetilde F_N$: Rejected $H^{FUR}_{0}$  & 0.999&   0.988&    0.904&    0.619&  0.241&    0.088\\
$\widehat T_{LV1}$: Rejected $H^{FUR}_{0}$  &  1.000 &   0.998&   0.927&    0.150&    0.135&  0.919
\\
$\widehat T_{LV10}$: Rejected $H^{FUR}_{0}$  &  0.891 &   0.699&   0.416&    0.220&    0.101& 0.040
\\
\hline $X^{(d,0,1)}$ &&&&&&\\
\hline \hline
$\widetilde F_N$: Rejected $H^{FUR}_{0}$  & 0.999&    0.990&   0.937&    0.800&    0.421&    0.109\\
$\widehat T_{LV1}$: Rejected $H^{FUR}_{0}$  &  1.000 &  1.000& 1.000  &  0.987 &    0.638&   0.132
\\
$\widehat T_{LV10}$: Rejected $H^{FUR}_{0}$  &  0.981 &   0.902 &   0.635&    0.358&    0.170&   0.072
\\
\hline
\hline
\end{tabular}\\
~\\
\vspace{3mm}
\begin{tabular}{|c|c|c|c|c|c|c|}
\hline\hline
$N=5000$ &&&&&&\\
ARFIMA$(0,d,0)$ &$d=0.5$ &$d=0.6$ &$d=0.7$ &$d=0.8$ &$d=0.9$ &$d=1$ \\
\hline \hline
$\widetilde F_N$: Rejected $H^{FUR}_{0}$  &  1.000  &  1.000  &  1.000  &  0.996   & 0.941  &  0.145 \\
$\widehat T_{DGM}$: Rejected $H^{FUR}_{0}$  &  1.000   & 1.000 &   1.000 &   1.000  &  1.000  &  0.059
\\
$\widehat T_{LV1}$: Rejected $H^{FUR}_{0}$  &  1.000   &   1.000  & 1.000  &   1.000  & 1.000  & 0.055
\\
$\widehat T_{LV10}$: Rejected $H^{FUR}_{0}$  &  1.000   & 1.000  & 1.000  &    0.997&    0.576& 0.049
\\
\hline
ARFIMA$(1,d,0)$ &&&&&& \\
\hline \hline
$\widetilde F_N$: Rejected $H^{FUR}_{0}$  &  1.000& 1.000&   0.998&  0.994&    0.904&   0.105\\
$\widehat T_{LV1}$: Rejected $H^{FUR}_{0}$  & 1.000 & 1.000&  1.000&   1.000&  1.000&  0.054
\\
$\widehat T_{LV10}$: Rejected $H^{FUR}_{0}$  &  1.000 &  1.000&  1.000&    0.999&    0.652&  0.056
\\
\hline
ARFIMA$(1,d,1)$ &&&&&& \\
\hline \hline
$\widetilde F_N$: Rejected $H^{FUR}_{0}$  &  1.000& 1.000&   0.998&  0.989&    0.808&   0.072\\
$\widehat T_{LV1}$: Rejected $H^{FUR}_{0}$  &  1.000 & 1.000&  1.000&  0.884&    0.965&   1.000
\\
$\widehat T_{LV10}$: Rejected $H^{FUR}_{0}$  &  1.000 & 1.000& 1.000&    0.988&    0.488&  0.051
\\
\hline $X^{(d,0,1)}$ &&&&&&\\
\hline \hline
$\widetilde F_N$: Rejected $H^{FUR}_{0}$  &  1.000&    1.000&    1.000&    0.998&    0.941&    0.122\\
$\widehat T_{LV1}$: Rejected $H^{FUR}_{0}$  &  1.000 &   1.000&   1.000&   1.000& 0.999 &   0.269
\\
$\widehat T_{LV10}$: Rejected $H^{FUR}_{0}$  &  1.000 &   1.000&   1.000&    0.983&    0.493&  0.058
\\
\hline
\hline
\end{tabular}
\end{center}
}
\caption{{ \small \label{Table9}Comparisons of Fractional Unit Roots  tests, as $1000$ independent Monte Carlo experiment replications, of processes for several values of $d$ and $N$. Note that the AR parameter of the ARFIMA$(1,d,0)$ is $0.5$ and the AR and MA parameters of ARFIMA$(1,d,1)$ are respectively $-0.3$ and $0.7$. The accuracy of tests is measured by the rejection probabilities.}
}
\end{table}
~\\
\underline{{\bf Conclusions of simulations:}} If $\widehat T_{DGM}$ and $\widehat T_{LV1}$ provide extremely convincing results for ARFIMA$(0,d,0)$ processes, $\widehat T_{LV1}$ is still very accurate for ARFIMA$(1,d,0)$ processes. From its definition, the fractional unit roots $\widehat T_{LV1}$ can not be used fruitfully for ARFIMA$(1,d,1)$ or $X^{(d,0,1)}$ processes but it can clearly be applied to a more general class of processes than $\widehat T_{DGM}$.\\
The same for $\widehat T_{LV10}$ which can be applied likely to a more general class of processes than $\widehat T_{LV1}$. However, if the results obtained here for ARFIMA$(1,d,1)$ or $X^{(d,0,1)}$ are satisfying and indicate that this test could be applied when $N=500$ and $N=5000$, this is not the case for $N=200$ for all the considered processes because this test requires the estimation of too many parameters. Moreover, $\widehat T_{LV10}$ can not theoretically be applied to ARFIMA$(1,d,1)$ or $X^{(d,0,1)}$ processes and additional simulations realized with $N=500000$ indicate a rejection probability $\simeq 0.16$ for ARFIMA$(1,d,1)$ processes when $d=1$ (instead of $0.05$). However, when $N=500000$, a user could probably chose $\widehat T_{LV20}$ or $\widehat T_{LV50}$ which would provide satisfying results. \\
The fractional unit roots test $\widetilde F_N$ constructed from $\widetilde {d}_N^{(MIR)}$ does not have these drawbacks: this is a data-driven test and it can be applied to a large family of fractional processes. Hence, the results of simulations obtained with $\widetilde F_N$ are satisfying (even if they are less efficient for specific processes than those obtained with $\widehat T_{LV10}$ which requires the knowledge of the $AR$ component). Even if this is a semiparametric test, the results obtained for $N=200$ are reasonable. However, note that the rejection probability of $\widetilde F_N$ for $d=0.9$ is much bigger than for $d=1$ when  $N=200$ and $N=500$. Hence one could size adjust this test to get a better performance.
		
\section{Conclusion}\label{conclu}
The adaptive data-driven memory parameter estimator $\widetilde {d}_N^{(MIR)}$ proposed in this paper has a lot of advantages. Firstly, for any process belonging to the set $\big (\mbox{IG}(d,\beta)\big ) _{-0.5 < d <1.25, \, 0<\beta\leq 2}$, it follows a CLT with a convergence rate reaching the minimax convergence rate (up to a multiplicative logarithm term) and this CLT is obtained without any choice of bandwidth or trimming parameter. Secondly, the numerical performances of this estimator are often better than those of the most accurate semiparametric memory parameter estimators, especially in case of trended processes (the robustness of IR estimator was already established in Surgailis {\it et al.}, 2008). Finally, data-driven stationarity and fractional unit roots tests are constructed from this estimator and they provide accurate competitive results with respect to classical unit roots or fractional unit roots tests. Improving the performance of those tests could be an interesting task.  

An asymptotic study of these new estimator and tests for linear processes could be an interesting extension of this paper. However, this requires to proof a multidimensional CLT theorem for a non-polynomial function of a multidimensional linear process which is a difficult result to be established.

\section{Proofs}\label{proofs}
Two technical lemmas are first established:
\begin{lemma}\label{Taqqup31}
For all $\lambda >0$
\begin{enumerate}
\item For $a\in(0,2 )$, $\displaystyle\frac{2}{|\lambda|^{a-1}}\int_{0}^{\infty}\frac{\sin(\lambda x)}{x^{a}}dx=\frac{4a}{2^{a}|\lambda|^{a}}\int_{0}^{\infty}\frac{\sin^{2}(\lambda x)}{x^{a+1}}dx=\frac{~\pi~}{\Gamma(a)\sin(\frac{a \pi}{2})}$;
\item For $b\in(-1,1 )$, $\displaystyle \frac{1}{2^{1-b}-1}\int_{0}^{\infty}\frac{\sin^{4}(\lambda x)}{x^{4-b}}dx=\frac{16}{-15+6\cdot2^{3-b}-3^{3-b}}\times \int_{0}^{\infty}\frac{\sin^{6}(\lambda x)}{x^{4-b}}dx=\frac{~2^{3-b}|\lambda|^{3-b}~\pi~}{4~\Gamma(4-b)\sin(\frac{(1-b) \pi}{2})}$;
\item For $b\in(1,3 )$, $\displaystyle \frac{1}{1-2^{1-b}}\int_{0}^{\infty}\frac{\sin^{4}(\lambda x)}{x^{4-b}}dx=\frac{16}{15-6\cdot2^{3-b}+3^{3-b}}\times \int_{0}^{\infty}\frac{\sin^{6}(\lambda x)}{x^{4-b}}dx=\frac{~2^{3-b}|\lambda|^{3-b}~\pi~}{4~\Gamma(4-b)\sin(\frac{(3-b) \pi}{2})}$.
    \end{enumerate}
\end{lemma}
\begin{proof}
These equations are given or deduced (using decompositions of $\sin^j(\cdot)$ and integration by parts) from (see Doukhan {\it et al.},~p. 31).\\
\end{proof}
\begin{lemma}\label{lemma46}
For $j=4,6$, denote
\begin{equation}\label{Jjdef}
J_j(a,m):=\int_{0}^{\pi}x^{a}\frac{\sin^{j}(\frac{mx}{2})}{\sin^{4}(\frac{x}{2})}dx.
\end{equation}
Then, we have the following expansions when $m \to \infty$:
\begin{eqnarray}\label{Jj}
 J_{j}(a,m)=\left \{ \begin{array}{ll}
                      {C}_{j1}(a)\,m^{3-a}+O\big(m^{1-a}\big) & \mbox{if $-1<a<1$}\\
{C}'_{j1}(1)\,m^{3-a}+O\big (\log(m)\big) & \mbox{if $a=1$} \\
{C}'_{j1}(a)\,m^{3-a}+O\big (1\big)& \mbox{if $1<a<3$} \\
{C}'_{j2}(3)\,\log(m)+O\big (1\big)& \mbox{if $a=3$} \\
{C''}_{j1}(a)+O\big(m^{-((a-3)\wedge2)})& \mbox{if $a>3$}
                     \end{array}
\right .
\end{eqnarray}
with the following real constants (which do not vanish for any $a$ on the corresponding set):
\begin{eqnarray*}\label{AllConstantTild}
&\bullet & {C}_{41}(a):=\frac{4~~\pi(1-\frac{2^{3-a}}{4})}{(3-a)\Gamma(3-a)\sin(\frac{(3-a)\pi}{2})}\quad  \mbox{and}\quad {C}_{61}(a):=\frac{\pi(15-6\cdot2^{3-a}~+3^{3-a})}{4(3-a)\Gamma(3-a)\sin(\frac{(3-a)\pi}{2})}\\
&\bullet & {C}'_{41}(a):=\Big(\frac{6}{3-a}\textbf{1}_{\{1\leq a<3\}}+16\int_{0}^{1}\frac{\sin^{4}(\frac{y}{2})}{y^{4-a}}dy+2\int_{1}^{\infty}\frac{1}{y^{4-a}}\Big(-4\cos(y)+\cos(2y)\Big)dy\Big)\\
&& \mbox{and} \quad {C}'_{61}(a):=\Big[16\int_{0}^{1}\frac{\sin^{6}(\frac{y}{2})}{y^{4-a}}dy+\frac{5}{3-a}\textbf{1}_{\{1\leq a<3\}}+\frac{1}{2}\int_{1}^{\infty}\frac{1}{y^{4-a}}\Big(-15\cos(y)+6\cos(2y)-\cos(3y)\Big)dy\Big]\\
&\bullet& {C}'_{42}(a):=\Big(6\cdot\textbf{1}_{\{a=3\}}+\textbf{1}_{\{a=1\}}\Big)\quad  \mbox{and}\quad
{C}'_{62}(a):=\Big(5\cdot\textbf{1}_{\{a=3\}}+\frac{5}{6}\cdot\textbf{1}_{\{a=1\}}\Big)\\
&\bullet& {C}''_{41}(a):=\frac{3}{8}\int_{0}^{\pi}\frac{x^{a}}{\sin^{4}(\frac{x}{2})}dx\quad  \mbox{and}\quad
{C}''_{61}(a):=\frac{5}{16}\int_{0}^{\pi}\frac{x^{a}}{\sin^{4}(\frac{x}{2})}dx.
\end{eqnarray*}
\end{lemma}
\begin{proof} The proof of these expansions follows the steps than those of Lemma 5.1 in Bardet and Dola (2012). Hence we write for $j=4,6$,
\begin{eqnarray}\label{Jjam}
J_j(a,m)&=&\widetilde{J}_{j}(a,m)+ \int_{0}^{\pi}x^{a}\sin^{j}(\frac{mx}{2})\frac{1}{(\frac{x}{2})^{4}}dx+ \int_{0}^{\pi}x^{a}\sin^{j}(\frac{mx}{2})\frac{2}{3}\frac{1}{(\frac{x}{2})^{2}}dx
\end{eqnarray}
with
\begin{eqnarray*}
\widetilde{J}_{j}(a,m)&:=&\int_{0}^{\pi}x^{a}\sin^{j}(\frac{mx}{2})\big(\frac{1}{\sin^{4}(\frac{x}{2})}-\frac{1}{(\frac{x}{2})^{4}}-\frac{2}{3}\frac{1}{(\frac{x}{2})^{2}}\big)dx.
\end{eqnarray*}
The expansions when $m\to \infty$ of both the right hand sided integrals in (\ref{Jjam}) are obtained from Lemma \ref{Taqqup31}. It remains to obtain the expansion of $\widetilde{J}_{j}(a,m)$. Then, using classical trigonometric and Taylor expansions:
\begin{eqnarray*}
\sin^{4}(\frac{y}{2})&=& \frac 1 8 \big ( 3-4\cos(y)+\cos(2y) \big )\quad \mbox{and}\quad \frac{1}{\sin^{4}(y)}-\frac{1}{y^{4}}-\frac{2}{3}\frac{1}{y^{2}}\sim \frac{11}{45}\quad (y \to 0)\\
\sin^{6}(\frac{y}{2})&=&\frac 1 {32} \big (10-15\cos(y)+6\cos(2y)-\cos(3y) \big )\quad \mbox{and}\quad \frac{1}{y^{5}}+\frac{1}{3}\frac{1}{y^{3}}-\frac{\cos(y)}{\sin^{5}(y)}\sim \frac{31}{945}\, y\quad (y \to 0),
\end{eqnarray*}
the expansions of $\widetilde{J}_{j}(a,m)$ can be obtained.

Numerical experiments show that ${C}''_{41}(a)\neq0$, ${C}''_{61}(a)\neq0$, ${C}''_{42}(a)\neq0$ and ${C}''_{62}(a)\neq0$.
\end{proof}

\begin{proof}[Proof of Proposition \ref{MCLT}]
This proposition is based on results of Surgailis {\it et al.} (2008) and was already proved in Bardet et Dola (2012) in the case $-0.5<d<0.5$. \\
{\it Mutatis mutandis}, the case $0.5<d<1.25$ can be treated exactly following the same steps. \\
The only new proof which has to be established concerns the case $d=0.5$ since Surgailis {\it et al.} (2008) do not provide a CLT satisfied by the (unidimensional) statistic $IR_N(m)$ in this case. Let $Y_{m}(j)$ the standardized process defined Surgailis {\it et al.} (2008). Then, for $d=0.5$,  \begin{eqnarray*}\label{Rmj}
\forall j \geq 1,\quad |\gamma_{m}(j)|=\big |\E \big ( Y_{m}(j) Y_{m}(0) \big ) \big | = \frac{2}{V^{2}_{m}} \Big |\int_{0}^{\pi}~\cos(jx)~x\big(c_{0}+O(x^{\beta})\big)\frac{\sin^{4}(\frac{mx}{2})}{\sin^{4}(\frac{x}{2})}dx \Big |.
\end{eqnarray*}
Denote $\gamma_{m}(j)=\rho_{m}(j)=\frac{2}{V^{2}_{m}}\big(I_{1}+I_{2}\big)$ as in (5.39) of Surgailis {\it et al.} (2008). The expansion (2.20) of  Surgailis {\it et al.} (2008)  remains true for $d=0.5$ and therefore $V^2_m\sim c_0V(0.5)m^2$ when $m\to \infty$. The same for the  inequality (5.42) when $d=0.5$ and thus $|I_{1}|\leq C\, m^{4}j^{-2}$. Finally, when $d=0.5$, we still have $I_2 = j^{-1} \sum_{q=1} ^{j/2} I_2(q)$ with $|I_2(q)|\leq C m^4 j^{-1}$ when $1 \leq q \leq j/m$ and $|I_2(q)|\leq C q^{-4}j^3$ when $j/m \leq q \leq j$ (see  details p. 536-537 of Surgailis {\it et al.}, 2008). Then, the inequality (5.41) remains true for $d=0.5$ and since we consider here $j \geq m$,
\begin{multline*}
 |I_{2}|\leq C\, m^{3}j^{-1}\quad  \Longrightarrow \quad |I_{1}+I_{2}|\leq C\, m^{3}j^{-1} \quad
\Longrightarrow \quad  |\gamma_{m}(j)|=|\rho_{m}(j)|\leq\frac{2}{V^{2}_{m}}\big(|I_{1}+I_{2}|\big)\leq C\, \frac{m}{j}.
\end{multline*}
Now let $\eta_{m}(j):=\frac{|Y_{m}(j)+Y_{m}(j+m)|}{|Y_{m}(j)|+|Y_{m}(j+m)|}:=\psi\Big(Y_{m}(j),Y_{m}(j+m)\Big)$. The Hermite rank of the function $\psi$ is $2$ and therefore the equation (5.23) of Surgailis {\it et al.} (2008) obtained from Lemma 1 of Arcones (1994) remains valid. Hence:
$$
\big |\Cov(\eta_{m}(0),\eta_{m}(j))\big |\leq C\frac{m^{2}}{j^{2}},
$$
 and then the equations (5.28-5.31) of Surgailis {\it et al.} (2008) remain valid for all $d\in[0.5,1.25)$. Then for $d=0.5$,
\begin{eqnarray} \label{CLTuni}
\sqrt{\frac{N}{m}}\, \Big(IR_N(m)-\E \big [ IR_N(m)\big ]\Big)\limiteloiNm \mathcal{N}\big(0,\sigma^{2}(0.5)\big),
\end{eqnarray}
with $\sigma^{2}(0.5)\simeq (0.2524)^{2}$.\\
~\\
To establish the multidimensional CLT for $0.5\leq d \leq 1.25$ from \eqref{CLTuni} and unidimensional CLT of Surgailis {\it et al} (2008) for $0.5 < d \leq 1.25$,  we can reproduce exactly  the steps 1 and 2 of the proof in Proposition 2.1 of Bardet and Dola (2012).
\end{proof}

\begin{proof}[Proof of Property \ref{devEIR}]
As in Surgailis {\it et al} (2008), we can write:
\begin{eqnarray} \label{RmVm}
\nonumber \E \big [ IR_N(m) \big ]=\E\big(\frac{|Y^{0}+Y^{1}|}{|Y^{0}|+|Y^{1}|}\big)=\Lambda(\frac{R_{m}}{V_{m}^{2}}) \quad
\mbox{with}\quad
\frac{R_{m}}{V_{m}^{2}}:=1-2 \, \frac{\int_{0}^{\pi}f(x)\frac{\sin^{6}(\frac{mx}{2})}{\sin^{2}(\frac{x}{2})}dx~}{\int_{0}^{\pi}f(x)
\frac{\sin^{4}(\frac{mx}{2})}{\sin^{2}(\frac{x}{2})}dx}.
\end{eqnarray}
Therefore an expansion of $R_{m}/V_{m}^{2}$ provides an expansion of $\E \big [ IR_N(m) \big ]$ when $m\to \infty$.\\
~\\
{\bf Step 1} Let $f$ satisfy Assumption $IG(d,\beta)$. Then we are going to establish that there exist positive real numbers ${C}_{1}$, ${C}_{2}$ and ${C}_{3}$ specified in (\ref{Ctild1}), (\ref{Ctild2}) and (\ref{Ctild3}) such that for $0.5\leq d<1.5$ and with $\rho(d)$ defined in (\ref{DefinitionRhod}),
\begin{eqnarray*}
&1.& \quad \mbox{if $ \beta < 2d-1$,} \quad  \frac{R_{m}}{V_{m}^{2}}=\rho(d)+{C}_{1}(2-2d,\beta)m^{-\beta}+O\Big(m^{-2} +m^{-2\beta}\Big); \\
&2.&\quad \mbox{if $ \beta = 2d-1$,} \quad  \frac{R_{m}}{V_{m}^{2}}=\rho(d)+{C}_{2}(2-2d,\beta)m^{-\beta}+O\Big(m^{-2}+m^{-2-\beta}\log(m)+m^{-2\beta} \Big); \\
&3.&\quad \mbox{if $ 2d-1<\beta <2d+1$,} \quad  \frac{R_{m}}{V_{m}^{2}}=\rho(d)+{C}_{3}(2-2d,\beta)m^{-\beta}+O\Big(m^{-\beta-\epsilon}+m^{-2d-1}\log(m)+m^{-2\beta}\Big);\\
&4.&\quad\mbox{if $\beta =2d+1$,} \quad  \frac{R_{m}}{V_{m}^{2}}=\rho(d)+O\Big(m^{-2d-1}~\log(m)+m^{-2} \Big).
\end{eqnarray*}
Under Assumption $IG(d,\beta)$ and with $J_j(a,m)$ defined in (\ref{Jj}) in Lemma \ref{lemma46}, it is clear that,
$$
\frac{R_{m}}{V_{m}^{2}}=1-2\, \frac{J_6(2-2d,m)+ \frac{c_{1}}{c_{0}}J_6(2-2d+\beta,m) +O(J_6(2-2d+\beta+\varepsilon))}
{J_4(2-2d,m)+ \frac{c_{1}}{c_{0}}J_4(2-2d+\beta,m) +O(J_4(2-2d+\beta+\varepsilon))},
$$
since $\displaystyle \int_{0}^{\pi}O(x^{2-2d+\beta+\varepsilon})\frac{\sin^{j}(\frac{mx}{2})}{\sin^{2}(\frac{x}{2})}dx=O(J_j(2-2d+\beta+\varepsilon))$. Now using the results of Lemma \ref{lemma46} and constants ${C}_{j\ell}$, ${C}'_{j\ell}$ and ${C}''_{j\ell}$, $j=4,\, 6$, $\ell=1,2$ defined in Lemma \ref{lemma46},\\
~\\
1. Let $0<\beta<2d-1<2$, {\it i.e.} $-1<2-2d+\beta<1$. Then
\begin{eqnarray*}
\frac{R_{m}}{V_{m}^{2}}& \hspace{-3mm}=&\hspace{-3mm} 1\hspace{-1mm} -\hspace{-1mm}2\, \frac{
{C}_{61}(2-2d)~m^{1+2d}\hspace{-1mm}+\hspace{-1mm}O\big(m^{2d-1}\big)+\hspace{-1mm}\frac{c_{1}}{c_{0}}{C}_{61}(2-2d+\beta)m^{1+2d-\beta}\hspace{-1mm}+\hspace{-1mm}O\big(m^{2d-1-\beta}\big)}
{{C}_{41}(2-2d)m^{1+2d}\hspace{-1mm}+\hspace{-1mm}O\big(m^{2d-1}\big)+\hspace{-1mm}\frac{c_{1}}{c_{0}} {C}_{41}(2-2d+\beta)m^{1+2d-\beta}\hspace{-1mm}+\hspace{-1mm}O\big(m^{2d-1-\beta}\big)}\\
&\hspace{-3mm}=&\hspace{-3mm}1\hspace{-1mm}-\hspace{-1mm}\frac{2}{{C}_{41}(2-2d)}\Big[{C}_{61}(2-2d)\hspace{-1mm}+\hspace{-1mm}\frac{c_{1}}{c_{0}}{C}_{61}(2-2d+\beta)m^{-\beta}\Big]\Big[1\hspace{-1mm}-\hspace{-1mm}\frac{c_{1}}{c_{0}}\frac{{C}_{41}(2-2d+\beta)}{{C}_{41}(2-2d)}m^{-\beta}\Big]\hspace{-1mm}+\hspace{-1mm}O\big(m^{-2}\big)\\
&\hspace{-3mm}=&\hspace{-3mm}1\hspace{-1mm}-\hspace{-1mm}\frac{2{C}_{61}(2-2d)}{{C}_{41}(2-2d)}\hspace{-1mm}+\hspace{-1mm}2\frac{c_{1}}{c_{0}}\Big[\frac{{C}_{61}(2-2d){C}_{41}(2-2d+\beta)}{{C}_{41}(2-2d){C}_{41}(2-2d)}\hspace{-1mm}-\hspace{-1mm}\frac{{C}_{61}(2-2d+\beta)}{{C}_{41}(2-2d)}\Big]m^{-\beta}\hspace{-1mm}+\hspace{-1mm}O\big(m^{-2}+m^{-2\beta}\big).
\end{eqnarray*}
As a consequence,,
\begin{multline}\label{Ctild1}
\frac{R_{m}}{V_{m}^{2}}=\rho(d)~+~{C}_{1}(2-2d,\beta)~~m^{-\beta}+~O\Big(m^{-2}+m^{-2\beta} \Big)\quad (m\to \infty),\quad \mbox{with $0<\beta<2d-1<2$ and} \\
{C}_{1}(2-2d,\beta):=2 \, \frac{c_{1}}{c_{0}}\frac 1 {{C}^2_{41}(2-2d)}\big[{C}_{61}(2-2d){C}_{41}(2-2d+\beta)-{C}_{61}(2-2d+\beta)
{C}_{41}(2-2d)\big],
\end{multline}
and numerical experiments proves that $ {C}_{1}(2-2d,\beta)/c_1$ is negative for any $d \in (0.5,1.5)$ and $\beta>0$. \\
~\\
2. Let $\beta=2d-1$, {\it i.e.} $2-2d+\beta=1$. Then,
\begin{eqnarray*}
\frac{R_{m}}{V_{m}^{2}}& \hspace{-3mm}=&\hspace{-3mm} 1\hspace{-1mm} -\hspace{-1mm}2\, \frac{
{C}_{61}(2-2d)~m^{1+2d}\hspace{-1mm}+\hspace{-1mm}O\big(m^{2d-1}\big)+\hspace{-1mm}\frac{c_{1}}{c_{0}}{C'}_{61}(1)m^{1-2d}\hspace{-1mm}+\hspace{-1mm}O\big(\log(m)\big)}
{{C}_{41}(2-2d)m^{1+2d}\hspace{-1mm}+\hspace{-1mm}O\big(m^{2d-1}\big)+\hspace{-1mm}\frac{c_{1}}{c_{0}} {C'}_{41}(1)m^{1-2d}\hspace{-1mm}+\hspace{-1mm}O\big(\log(m)\big)}\\
&\hspace{-3mm}=&\hspace{-3mm}1\hspace{-1mm}-\hspace{-1mm}\frac{2}{{C}_{41}(2-2d)}\Big[{C}_{61}(2-2d)\hspace{-1mm}+\hspace{-1mm}\frac{c_{1}}{c_{0}}{C'}_{61}(1)m^{1-2d}\Big]\Big[1\hspace{-1mm}-\hspace{-1mm}\frac{c_{1}}{c_{0}}\frac{{C'}_{41}(1)}{{C}_{41}(2-2d)}m^{1-2d}\Big]\hspace{-1mm}+\hspace{-1mm}O\big(m^{-2}+m^{-2d-1}\log(m)\big)\\
&\hspace{-3mm}=&\hspace{-3mm}1\hspace{-1mm}-\hspace{-1mm}\frac{2{C}_{61}(2-2d)}{{C}_{41}(2-2d)}\hspace{-1mm}+\hspace{-1mm}2\frac{c_{1}}{c_{0}}\Big[\frac{{C}_{61}(2-2d){C'}_{41}(1)}{{C}_{41}(2-2d){C}_{41}(2-2d)}\hspace{-1mm}-\hspace{-1mm}\frac{{C'}_{61}(1)}{{C}_{41}(2-2d)}\Big]m^{1-2d}\hspace{-1mm}+\hspace{-1mm}O\big(m^{-2}+m^{-2d-1}\log(m)+m^{2-4d}\big).
\end{eqnarray*}
As a consequence,
\begin{multline}\label{Ctild2}
\frac{R_{m}}{V_{m}^{2}}=\rho(d)~+~{C}_{2}(2-2d,\beta)~~m^{-\beta}+~O\Big(m^{-2}+m^{-2-\beta}\log(m)+m^{-2\beta} \Big)\quad (m\to \infty),\quad \mbox{with $0<\beta=2d-1<2$ and } \\
{C}_{2}(2-2d,\beta):=2 \, \frac{c_{1}}{c_{0}}\frac 1 {{C}^2_{41}(2-2d)}\big[{C}_{61}(2-2d){C'}_{41}(1)-{C'}_{61}(1)
{C}_{41}(2-2d)\big],
\end{multline}
and numerical experiments proves that $ {C}_{2}(2-2d,\beta)/c_1$ is negative for any $d \in [0.5,1.5)$ and $\beta>0$. \\\\
3.  Let $2d-1<\beta<2d+1$,  {\it i.e.} $1<2-2d+\beta<3$. Then,
\begin{eqnarray*}
\frac{R_{m}}{V_{m}^{2}}& \hspace{-3mm}=&\hspace{-3mm} 1\hspace{-1mm} -\hspace{-1mm}2\, \frac{
{C}_{61}(2-2d)m^{1+2d}\hspace{-1mm}+\hspace{-1mm}\frac{c_{1}}{c_{0}}{C'}_{61}(2-2d+\beta)m^{1+2d-\beta}\hspace{-1mm}+\hspace{-1mm}O\big(m^{1+2d-\beta-\epsilon}+\log(m)\big)}
{{C}_{41}(2-2d)m^{1+2d}\hspace{-1mm}+\hspace{-1mm}\frac{c_{1}}{c_{0}} {C'}_{41}(2-2d+\beta)m^{1+2d-\beta}\hspace{-1mm}+\hspace{-1mm}O\big(m^{1+2d-\beta-\epsilon}+m^{-2d-1}\log(m)\big)}\\
&\hspace{-3mm}=&\hspace{-3mm}1\hspace{-1mm}-\hspace{-1mm}\frac{2}{{C}_{41}(2-2d)}\Big[{C}_{61}(2-2d)\hspace{-1mm}+\hspace{-1mm}\frac{c_{1}}{c_{0}}{C'}_{61}(2-2d+\beta)m^{-\beta}\Big]\Big[1\hspace{-1mm}-\hspace{-1mm}\frac{c_{1}}{c_{0}}\frac{{C'}_{41}(2-2d+\beta)}{{C}_{41}(2-2d)}m^{-\beta}\Big]\hspace{-1mm}+\hspace{-1mm}O\big(m^{-\beta-\epsilon}+m^{-2d-1}\log(m)\big)\\
&\hspace{-3mm}=&\hspace{-3mm}1\hspace{-1mm}-\hspace{-1mm}\frac{2{C}_{61}(2-2d)}{{C}_{41}(2-2d)}\hspace{-1mm}+\hspace{-1mm}2\frac{c_{1}}{c_{0}}\Big[\frac{{C}_{61}(2-2d){C'}_{41}(2-2d+\beta)}{{C}_{41}(2-2d){C}_{41}(2-2d)}\hspace{-1mm}-\hspace{-1mm}\frac{{C'}_{61}(2-2d+\beta)}{{C}_{41}(2-2d)}\Big]m^{-\beta}\hspace{-1mm}+\hspace{-1mm}O\big(m^{-\beta-\epsilon}+m^{-2d-1}\log(m)\big).
\end{eqnarray*}
As a consequence,
\begin{multline}\label{Ctild3}
\frac{R_{m}}{V_{m}^{2}}=\rho(d)~+~{C}_{3}(2-2d,\beta)~~m^{-\beta}+~O\Big(m^{-\beta-\epsilon}+m^{-2d-1}\log(m)+m^{-2\beta} \Big)\quad (m\to \infty),\quad \mbox{and} \\
{C}_{3}(2-2d,\beta):=2 \, \frac{c_{1}}{c_{0}}\frac 1 {{C}^2_{41}(2-2d)}\big[{C}_{61}(2-2d){C'}_{41}(2-2d+\beta)-{C'}_{61}(2-2d+\beta)
{C}_{41}(2-2d)\big],
\end{multline}
and numerical experiments proves that $ {C}_{3}(2-2d,\beta)/c_1$ is negative for any $d \in [0.5,1.5)$ and $\beta>0$. \\
~\\
4. Let $\beta=2d+1$. Then,
Once again with Lemma \ref{lemma46}:
\begin{eqnarray*}
\frac{R_{m}}{V_{m}^{2}}& \hspace{-3mm}=&\hspace{-3mm} 1\hspace{-1mm} -\hspace{-1mm}2\, \frac{
{C}_{61}(2-2d)~m^{1+2d}\hspace{-1mm}+\hspace{-1mm}O\big(m^{2d-1}\big)+\hspace{-1mm}\frac{c_{1}}{c_{0}}{C'}_{62}(3)\log(m)\hspace{-1mm}+\hspace{-1mm}O\big(1\big)}
{{C}_{41}(2-2d)m^{1+2d}\hspace{-1mm}+\hspace{-1mm}O\big(m^{2d-1}\big)+\hspace{-1mm}\frac{c_{1}}{c_{0}} {C'}_{42}(3)\log(m)\hspace{-1mm}+\hspace{-1mm}O\big(1\big)}\\
&\hspace{-3mm}=&\hspace{-3mm}1\hspace{-1mm}-\hspace{-1mm}\frac{2}{{C}_{41}(2-2d)}\Big[{C}_{61}(2-2d)\hspace{-1mm}+\hspace{-1mm}\frac{c_{1}}{c_{0}}{C'}_{62}(3)m^{-\beta}\log(m)\Big]\Big[1\hspace{-1mm}-\hspace{-1mm}\frac{c_{1}}{c_{0}}\frac{{C'}_{42}(3)}{{C}_{41}(2-2d)}m^{-\beta}\log(m)\Big]\hspace{-1mm}+\hspace{-1mm}O\big(m^{-2}+m^{-2d-1}\big)\\
&\hspace{-3mm}=&\hspace{-3mm}1\hspace{-1mm}-\hspace{-1mm}\frac{2{C}_{61}(2-2d)}{{C}_{41}(2-2d)}\hspace{-1mm}+\hspace{-1mm}2\frac{c_{1}}{c_{0}}\Big[\frac{{C}_{61}(2-2d){C'}_{42}(3)}{{C}_{41}(2-2d){C}_{41}(2-2d)}\hspace{-1mm}-\hspace{-1mm}\frac{{C'}_{62}(3)}{{C}_{41}(2-2d)}\Big]m^{-\beta}\log(m)\hspace{-1mm}+\hspace{-1mm}O\big(m^{-2}\big).
\end{eqnarray*}
As a consequence,
\begin{eqnarray}\label{C0C1}
\frac{R_{m}}{V_{m}^{2}}=\rho(d)~+~O\big(m^{-2d-1}~\log(m)+m^{-2} \big)\quad (m\to \infty),\quad \mbox{with $2<\beta=2d+1<4$.}
\end{eqnarray}
{\bf Step 2:} A Taylor expansion of $\Lambda(\cdot)$ around $\rho(d)$ provides:
\begin{eqnarray*}
\Lambda\Big(\frac{R_{m}}{V_{m}^{2}}\Big) \simeq \Lambda\big(\rho(d)\big)+\Big[\frac{\partial\Lambda}{\partial\rho}\Big](\rho(d))\Big(\frac{R_{m}}{V_{m}^{2}}-\rho(d)\Big)+\frac{1}{2} \, \Big[\frac{\partial^{2}\Lambda}{\partial\rho^{2}}\Big](\rho(d))\Big(\frac{R_{m}}{V_{m}^{2}}-\rho(d)\Big)^2.
\end{eqnarray*}
Note that numerical experiments show that $\displaystyle \Big[\frac{\partial\Lambda}{\partial \rho}\Big](\rho)>0.2$ for any $\rho\in (-1,1)$.
As a consequence, using the previous expansions of $R_{m}/V_{m}^{2}$ obtained in Step 1 and since $\E \big [IR_N(m)\big ]=\Lambda\big(R_{m}/V_{m}^{2}\big)$, then for all $0<\beta\leq 2$:
\begin{eqnarray*}
\E \big [IR_N(m)\big ]=\Lambda_{0}(d)+\left \{ \begin{array}{ll} c_1 \, {C'}_1(d,\beta)\,m^{-\beta}+O\big (m^{-2} +m^{-2\beta} \big )& \mbox{if}~~~\beta<2d-1 \\
 c_1 \, {C'}_2(d,\beta)\,m^{-\beta}+O\big (m^{-2}+m^{-2-\beta}\log m +m^{-2\beta} \big )& \mbox{if}~~~\beta=2d-1 \\
  c_1 \,{C'}_3(d,\beta)\,m^{-\beta}+O\big (m^{-\beta-\epsilon}+m^{-2d-1}\log m +m^{-2\beta} \big )& \mbox{if}~~~2d-1<\beta<2d+1 \\
O\big(m^{-2d-1}\log m+m^{-2}\big)& \mbox{if}~~~\beta=1+2d \\

\end{array}\right .
\end{eqnarray*}
with $C'_\ell(d,\beta)=\Big [\frac{\partial\Lambda}{\partial\rho}\Big](\rho(d)) \, C_\ell(2-2d,\beta)$ for $\ell=1,2,3$ and $ C_\ell$ defined in (\ref{Ctild1}), (\ref{Ctild2}) and (\ref{Ctild3}).
\end{proof}

\begin{proof}[Proof of Theorem \ref{cltnada}] Using Property \ref{devEIR}, if $m \simeq C\, N^\alpha$ with $C>0$ and $(1+2\beta)^{-1}  <\alpha<1$ then $\sqrt{N/m}\, \big (\E \big [IR_N(m)\big ] -\Lambda_{0}(d)\big ) \limiteN 0$ and it implies that the multidimensional CLT (\ref{TLC1}) can be replaced by
\begin{eqnarray}\label{TLC1bis}
\sqrt{\frac{N}{m}}\Big (IR_N(m_j)-\Lambda_{0}(d)\Big )_{1\leq j \leq p}\limiteloiN {\cal N}(0, \Gamma_p(d)).
\end{eqnarray}
It remains to apply the Delta-method with the function $\Lambda_0^{-1}$ to CLT (\ref{TLC1bis}). This is possible since the function $d \to \Lambda_0(d)$ is an increasing function such that $\Lambda_0'(d)>0$ and $\big (\Lambda_0^{-1})'(\Lambda_0(d))=1/\Lambda'_0(d)>0$ for all $d\in (-0.5,1.5)$. It achieves the proof of Theorem \ref{cltnada}.
\end{proof}

\begin{proof}[Proof of Proposition \ref{hatalpha}]
We use the proof of Proposition 2 in Bardet and Dola (2012). Indeed, this proof is only based on the definitions of $\widehat d_N(jN^\alpha)$, $\widetilde d_N(N^\alpha)$, $\widehat Q_N(\alpha)$ and $\widehat \alpha_N$ which are exactly the same here, and on the CLT satisfied by $\big(\widehat d_N(jN^\alpha)\big )_{1\leq j \leq p}$ The only difference is that we suppose here $0<\beta\leq 2$ and $0.5 \leq d < 1.25$ instead of $0<\beta$ and $-0.5< d <0.5$ in Bardet and Dola (2012). Then we consider only the case where $\beta \leq 2d+1$ and $\alpha^*=\frac 1 {(1+2\beta)\wedge (4d +3)}=\frac 1 {(1+2\beta)}$ in our framework.
\end{proof}

\begin{proof}[Proof of Theorem \ref{tildeD}]
This proof is exactly the same as the proof of Theorem 2 in  Bardet and Dola (2012).
\end{proof}

\medskip
\noindent
{\bf Acknowledgments.} The authors are grateful to both the referees for their very careful reading and many relevant
suggestions and corrections that strongly improve the content and the form of the paper.

\bibliographystyle{amsalpha}

\end{document}